\documentclass[psamsfonts,12pts]{amsart}

\usepackage{mathpazo}
\usepackage[utf8]{inputenc}
\usepackage[centering]{geometry}
\usepackage{amsmath}
\usepackage{amsthm}
\usepackage{amssymb}
\usepackage{amscd}
\usepackage{amsfonts}
\usepackage{amsbsy}
\usepackage{graphicx}
\usepackage[dvips]{psfrag}
\usepackage{array}
\usepackage{color}
\usepackage{epsfig}
\usepackage{url}
\usepackage{overpic}
\usepackage{epstopdf}
\usepackage[normalem]{ulem}
\usepackage{tikz}

\usepackage{caption}
\usepackage{subcaption}
\usepackage{multirow}
\usepackage{indentfirst}
\usepackage{mathrsfs}
\usepackage{hyperref}
\usepackage{placeins}
\usepackage{flafter}
\usepackage{tcolorbox}
\usepackage{amsmath}
\usetikzlibrary{arrows}

\usepackage[foot]{amsaddr}

\newcommand{\sgn}{\mathrm{sgn}}

\newtheorem {theorem} {Theorem}

\newtheorem {lemma}{Lemma}

\newtheorem {remark}{Remark}

\allowdisplaybreaks

\begin{document}
\renewcommand{\arraystretch}{1.5}
	
\title[On the existence of closed trajectories and pseudo-trajectories]{On the existence of closed trajectories and pseudo-trajectories for a family of third order differential equations}
\author[M. D. A. Caldas, R. M. Martins]
{Mayara D. A. Caldas$^{1}$, Ricardo M. Martins$^{1}$}
	
\address{$^1$ Departamento de Matem\'{a}tica, Universidade
Estadual de Campinas, Rua S\'{e}rgio Buarque de Holanda, 651, Cidade Universit\'{a}ria Zeferino Vaz, 13083--859, Campinas, SP, Brazil.} \email{mcaldas@ime.unicamp.br, rmiranda@unicamp.br}
	
\subjclass[2010]{34A36,37C29,37H20,34C28}
	
\keywords{piecewise smooth differential equations, averaging theory, closed trajectory}
	
\maketitle

{\bf Abstract.}  The goal of this article is to study the existence of closed trajectories for the differential equation $\dddot{z}+a\ddot{z}+b\dot{z}+abz=\varepsilon F(z,\dot{z},\ddot{z})$ in two situations. In the first situation, we consider $F(z,\dot{z},\ddot{z})=1$ and $b=\sgn(h(z,\dot{z},\ddot{z}))$, where $h(z,\dot{z},\ddot{z})=z^2+(\dot{z})^2+(\ddot{z})^2-1$. We show that the differential equation is equivalent to a piecewise smooth differential system that admits the unit sphere as the discontinuity manifold. We obtain conditions for the existence of a closed pseudo-trajectory in this case. In the second situation, we consider $\varepsilon \neq 0$ sufficiently small, $b>0$, and $F(z,\dot{z},\ddot{z})$ a $n$-degree polynomial. We show that the unperturbed differential equation has a family of isochronous periodic solutions filling an invariant plane. Then, we study the maximum number of limit cycles which bifurcate from this 2-dimensional isochronous using the averaging theory. Thus, within the same family, we have periodic solutions (in the case where the parameters create a smooth equation) and also pseudo-periodic solutions (in the case of Filippov systems). \\[3pt]
{\bf Keywords.} piecewise smooth differential equations, averaging theory, closed trajectory. \\[3pt]
{\small\bf AMS (MOS) subject classification:} 34A36, 37C29, 37H20, 34C28

\vskip.2in

\section{Introduction and Statement of the Results}

\noindent In this paper, we study the differential equation 
\begin{equation}\label{maineq}
	\dddot{z}+a\ddot{z}+b\dot{z}+abz=\varepsilon F(z,\dot{z},\ddot{z}).   
\end{equation}
A differential equation similar to this one was presented in the Chapter 2 of Barbashin's book \cite{Bseq}, having application in the study of the stability of automatic control systems with variable structure. 

The differential equation (\ref{maineq}) is equivalent to the following first-order differential system
\begin{equation}\left\{\label{sisper}
	\begin{array}{ccl}
		\dot{x} & = & y, \\
		\dot{y} & = & -ay-bx-abz+\varepsilon F(z,x,y),\\ 
		\dot{z} & = & x,
	\end{array}\right.
\end{equation}
where the dot denotes the derivative with respect to the time variable $t$. We are interested in studying the existence of closed trajectories for this differential system in two different situations. Hence, we will divide the paper into two parts.

In the first part, we consider $F(z,\dot{z},\ddot{z})=1$ and $b=sgn(h(z,\dot{z},\ddot{z}))$, where $h(z,\dot{z},\ddot{z})=z^2+(\dot{z})^2+(\ddot{z})^2-1$. Thus, the differential equation (\ref{maineq}) is equivalent to the piecewise smooth differential system
\begin{equation}\label{sdsp}
	Z_{X_-X_+}(x,y,z)=\left\lbrace 
	\begin{array}{cc}
		X_-(x,y,z), & x^2+y^2+z^2< 1,\\
		X_+(x,y,z), & x^2+y^2+z^2> 1,
	\end{array}\right.  
\end{equation}
where $X_-(x,y,z)=(y,-ay+x+az+\varepsilon,x)$ and $X_+(x,y,z)=(y,-ay-x-az+\varepsilon,x)$, which
admits $S^2=h^{-1}(\{0\})$, the unit sphere, as the discontinuity manifold. Over $S^2$ we assume that the dynamics of $Z_{X_-X_+}$ is provided by Filippov’s convention \cite{F}. We are interested in the existence of a closed trajectory for the differential system (\ref{sdsp}) that intersects $S^2$ in two points.

The existence of closed trajectories for piecewise smooth differential equations is an area of research that has been studied by several authors. For cases in which the piecewise smooth differential system belongs to $\mathbb{R}^2$, there are many works that determine the maximum number of limit cycles for a given class of vector fields \cite{LDTmcl, LXmcl, MR}, but there are still many open cases, such as the one in which the discontinuity manifold is a straight line \cite{JP3cl}. There are also studies in which the piecewise smooth differential system belongs to $\mathbb{R}^3$, for instance \cite{JM3D, LDM}, where the discontinuity manifold is a plane. Furthermore, one of the particular studies of discontinuous dynamical systems is the existence of pseudo-cycles which tends to appear due to the discontinuity of the system, intuitively seem to be crossing periodic orbits, but their orientation changes depending on the zone. This is the object of study in \cite{andrade2022homoclinic} and other articles.

Our main result for the first part is the following theorem.

\begin{theorem}\label{teo1}
	Consider $F(z,\dot{z},\ddot{z})=1$ and $b=\sgn(h(z,\dot{z},\ddot{z}))$, where 
$h(z,\dot{z},\ddot{z})$ $=z^2+(\dot{z})^2+(\ddot{z})^2-1$. If 
	$$\frac{|a|}{\sqrt{2}}<\varepsilon <|a| \quad \mbox{or} \quad -|a|< \varepsilon < -\frac{|a|}{\sqrt{2}},$$  
	then the system (\ref{sisper}) admits a pseudo-orbit.
\end{theorem}

In the second part, we deal with the differential equation (\ref{maineq}) considering $\varepsilon \neq 0$ sufficiently small, $b>0$ and $F(z,\dot{z},\ddot{z})$ a polynomial of degree $n$, which has a family of isochronous periodic solutions in an invariant plane when $\varepsilon=0$.

In general, when we consider a vector field $Z_0$ in $\mathbb{R}^3$, having a $2$-dimensional isochronous subset $\mathcal{I}$, and we take the vector field $Z_{\varepsilon}$ as an $\varepsilon$-perturbation of $Z$, that is, $Z_{\varepsilon}=Z_0+\varepsilon Z$. The natural questions we ask are: Does the vector field $Z_{\varepsilon}$ have limit cycles emerging from $\mathcal{I}$? How many? How do we compute them? These are the questions that we intend to answer for the differential equation that we are considering.

The averaging theory is the tool to study these questions and many works have been done using it, with the objective of studying the limit cycles bifurcating from the periodic orbits of a $k$-dimensional isochronous center contained in $\mathbb{R}^n$ with $k\leq n$. Among them, the cases where the $2$-dimensional isochronous subsets contained in $\mathbb{R}^3$ are a cylinder \cite{LMcilindro} and a torus \cite{JSJtoro}, and the cases where the $2$-dimensional isochronous subsets contained in $\mathbb{R}^4$ are planes \cite{JMplano, JMJplane}.

Our main result for the second part is the following theorem.

\begin{theorem}\label{teo2}
	Assume $b>0$. For $\varepsilon=0$ the plane $$\alpha=\left\lbrace (x,y,z) \in \mathbb{R}^3\; |\; y+bz=0 \right\rbrace$$ is invariant with respect to the system $(\ref{sisper})|_{\varepsilon=0}$, which behaves like a center in $\alpha$. Now, 
	considering $\varepsilon \neq 0$ sufficiently small, if $F(z,\dot{z},\ddot{z})$ is a polynomial of degree $n$ and 
	\begin{align*}
		\mathcal{F}(r) & =\frac{1}{2\pi (a^2+b)}\displaystyle\int_0^{\frac{2\pi}{\sqrt{b}}} \left(\sqrt{b} \cos \left(\sqrt{b} \theta \right)-a \sin \left(\sqrt{b} \theta \right)\right) \\
		& \quad F\left(\sqrt{b} r \sin \left(\sqrt{b} \theta \right),-b r \cos \left(\sqrt{b} \theta \right),r \cos \left(\sqrt{b} \theta \right)\right) \, d\theta
	\end{align*}
	does not vanish identically, then
	\begin{itemize}
		\item[a)]if $n$ is odd, $\frac{n-1}{2}$ is the maximum number of limit cycles of the system (\ref{sisper}) that can bifurcate from the periodic orbits of the plane $\alpha=\left\lbrace (x,y,z) \in \mathbb{R}^3\; |\; y+bz=0 \right\rbrace$ of the system $(\ref{sisper})|_{\varepsilon =0}$,
		\item[b)] if $n$ is even, $\frac{n-2}{2}$ is the maximum number of limit cycles of the system (\ref{sisper}) that can bifurcate from the periodic orbits of the plane $\alpha=\left\lbrace (x,y,z) \in \mathbb{R}^3\; |\; y+bz=0 \right\rbrace$ of the system $(\ref{sisper})|_{\varepsilon =0}$.
	\end{itemize}	
\end{theorem}

This paper is divided as follows. Section 2 presents some fundamental concepts required for the proof of the main results,such as the basic definitions about piecewise smooth vector fields according to Filippov’s convention, the averaging theory and the Descartes' Theorem. In Section \ref{prteo1}, we prove Theorem \ref{teo1} with the support of some lemmas that are also proved in the same section. In Section \ref{prteo2}, we prove Theorem \ref{teo2} following the same structure as in Section \ref{prteo1}.  

\section{Preliminaries}\label{resim}

We are going to present the basic concepts that are requisite in the proof of the main results of this paper.

\subsection{Piecewise smooth differential equations.}
\label{fili}

In this subsection we recall some concepts about piecewise smooth differential equations, according to the Filippov's convention \cite{F}. 

Let $X_+$ and $X_-$ be smooth vector fields defined in an open and convex subset $U\subset \mathbb{R}^3$ and, without loss of generality, assume that the origin belongs to $U$. Consider $h:U\rightarrow\mathbb{R}$ a function $\mathcal{C}^r$, with $r>1$, ($\mathcal{C}^r$ denotes the set of continuously differentiable functions of order $r$), having $0$ as a regular value, and let the curve  $\Sigma=h^{-1}(0)\cap U$ be a submanifold that divides the open set $U$ in two open sets,  
$$\Sigma^{+}=\{(x,y,z)\in U\;|\;h(x,y,z)>0\} \; \mbox{and} \; \Sigma^{-}=\{(x,y,z)\in U\;|\;h(x,y,z)<0\}.$$

A piecewise smooth vector field is defined in the form 
\begin{equation}\label{sist.fil}
	Z_{X_-X_+}(x,y,z)=\begin{cases}
		X_+(x,y,z),& (x,y,z) \in \Sigma^{+}, \\
		X_-(x,y,z), & (x,y,z) \in \Sigma^{-}.
	\end{cases}
\end{equation}
We assume that $X_+$ and $X_-$ are vector fields of class $\mathcal{C}^k$, with $k>1$, in $\overline{\Sigma^{+}}$ and $\overline{\Sigma^{-}}$, respectively, where $\overline{\Sigma^{\pm}}$ denotes the closure of $\Sigma^{\pm}$.    
We denote by $\mathcal{Z}^k$ the space of vector fields of this type, that can be taken as $\mathcal{Z}^k=\mathcal{X}^k\times \mathcal{X}^k$, where, by abuse of notation, $\mathcal{X}^k$ denotes the set of vector fields of class $\mathcal{C}^k$ defined in $\overline{\Sigma^{+}}$ and $\overline{\Sigma^{-}}$. We consider $\mathcal{Z}^k$ with the product topology $\mathcal{C}^k$.  

We are going to divide the discontinuity submanifold $\Sigma$ in the closure of three disjoint regions. For this, we define the Lie's derivative of $h$ with respect to the field $X_{\pm}$ in $p$ as 
$$X_{\pm}h(p)=\langle X_{\pm}(p), \nabla h(p)\rangle.$$
Thus the regions are classified according to the following:
\begin{itemize}
	\item[(a)] Crossing region:  $\Sigma^{c}=\{p \in \Sigma\;|\; X_+h(p)\cdot X_-h(p) >0\},$
	\item[(b)] Sliding region:  $\Sigma^{s}=\{p \in \Sigma\;|\; X_+h(p) <0,  X_-h(p) >0\},$
	\item[(c)] Escape region:  $\Sigma^{e}=\{p \in \Sigma\;|\; X_+h(p) >0,  X_-h(p) <0\}.$
\end{itemize}
These three regions are open subsets of $\Sigma$ in the induced topology and can have more than one convex component. 

When defining the regions above we aren't including the tangent points, that is, the points $p \in \Sigma$ for which $X_+h(p)=0$ or $X_-h(p)=0$. These points are in the boundaries of the regions $\Sigma^{c}$, $\Sigma^{s}$ and $\Sigma^{e}$, that are going to be denoted by $\partial \Sigma^{c}$, $\partial \Sigma^{s}$ and $\partial\Sigma{}^{e}$, respectively. Moreover, if a point $p\in \Sigma$ satisfies $X_{\pm}h(p)=0$ and $X_{\pm}^2f(p)\neq 0$, then we call $p$ a fold tangent point, and if a point $p\in \Sigma$ satisfies $X_{\pm}h(p)=X_{\pm}^2h(p)=0$ and $X_{\pm}^3h(p)\neq 0$ then we call $p$ a cuspid tangent point, where $X_{\pm}^kh(p)=\langle X_{\pm}(p), \nabla X_{\pm}^{k-1}h(p)\rangle$ for $k\geq 2$. 

To establish the dynamic given by a vector field $Z_{X_-X_+}$ in $U$, we need to define the local trajectory through a point $p\in U$, that is, we must define the flow $\varphi_z(t,p)$ of (\ref{sist.fil}).
If $p \in \Sigma_{}^{\pm}$, then the trajectory through $p$ is given by the fields $X_+$ and $X_-$ in the usual way. 

If $p \in \Sigma^{c}$, both the vector fields $X_-$ and $X_+$ point to $\Sigma^{+}$ or $\Sigma^{-}$ and, therefore, it is sufficient to concatenate the trajectories of $X$ and $Y$ that pass through $p$. 

If $p\in \Sigma^{s}\cup \Sigma^{e}$, we have that the vector fields point to opposite directions, thus, we can't concatenate the trajectories. In this case, the local orbit is given by the Filippov's convention. Hence, we define the sliding vector field
\begin{equation}\label{campo_deslizante}
	Z^{s}(p)=\frac{1}{X_-h(p)-X_+h(p)}(X_-h(p)X(p)-X_+h(p)X_-(p)).
\end{equation}  
Note that $Z^s$ represents the convex linear combination of $X_+(p)$ and $X_-(p)$ so that $Z^{s}$ is tangent to $\Sigma$. Moreover, its trajectories are contained in $\Sigma^{s}$ or $\Sigma^{e}$. Thus, the trajectory through $p$ is the trajectory defined by the sliding vector field in (\ref{campo_deslizante}).


\subsection{Averaging theory} 

In this subsection, we present a result about the averaging  theory \cite{BFL}, that we will use to prove Theorem \ref{teo2}. 

We consider the problem of the bifurcation of $T$-periodic orbits from differential systems 
\begin{equation}\label{aver}
	{\bf \dot{x}}(t)=F_0(t,{\bf x})+\varepsilon F_1(t,{\bf x})+\varepsilon^2 F_2(t,{\bf x},\varepsilon),
\end{equation}	
with $\varepsilon \in (-\varepsilon_0,\varepsilon_0)$ and $\varepsilon_0>0$ sufficiently small. The functions $F_0, F_1:\mathbb{R}\times \Omega \rightarrow \mathbb{R}^n$ and $F_2: \mathbb{R}\times \Omega \times (-\varepsilon_0,\varepsilon_0) \rightarrow \mathbb{R}^n$ are $\mathcal{C}^k$ functions, with $k\geq 2$, $T$-periodic in the first variable, and $\Omega$ is an open subset of $\mathbb{R}^n$. The main assumption is that the unperturbed system ${\bf \dot{x}}(t)=F_0 (t,{\bf x})$ has a $m$-dimensional submanifold, with $1\leq m \leq n$, comprised of $T$-periodic orbits. 

Let ${\bf x}(t,{\bf z})$ be the solution of the unperturbed system $(\ref{aver})|_{\varepsilon=0}$ such that ${\bf x}(0,{\bf z})={\bf z}$. We write the linearization of the unperturbed system along the periodic solution ${\bf x}(t,{\bf z})$ as
\begin{equation}\label{linear} 
	{\bf \dot{y}}(t)=D_{\bf x}F_0(t,{\bf x}(t,{\bf z})){\bf y}.
\end{equation}
In the following, we denote by $M_z(t)$ a fundamental matrix of the linear differential system (\ref{linear}) and by $\xi: \mathbb{R}^k \times \mathbb{R}^{n-k} \rightarrow \mathbb{R}^k$ the projections of $\mathbb{R}^n$ onto its first $k$ and $n-k$ coordinates, respectively, that is, $\xi(x_1,\dots, x_n)=(x_1,\dots x_k)$.

\begin{theorem}\label{averaging}
	Let $V \subset \mathbb{R}^k$ be an open and bounded subset and let $\beta_0 :\overline{V} \rightarrow \mathbb{R}^{n-k}$ be a $\mathcal{C}^k$ function. We assume that
	\begin{itemize}
		\item[(a)] $ \mathcal{Z}=\{{\bf z}_{\alpha} = (\alpha,\beta_0 ( \alpha))\;|\; \alpha \in \overline{V}\} \subset \Omega$ and that for each ${\bf z}_{\alpha} \in \mathcal{Z}$ the solution ${\bf x}(t,{\bf z}_{\alpha})$ of $(\ref{aver})|_{\varepsilon=0}$ is $T$-periodic, \\
		\item[(b)] for each ${\bf z}_{\alpha}\in \mathcal{Z}$ there is a fundamental matrix $M_{{\bf z}_{\alpha}}(t)$ of (\ref{linear}) such that the matrix $M_{{\bf z}_{\alpha}}^{-1}(0)-M_{{\bf z}_{\alpha}}^{-1}(T)$ has in the upper right corner the $k\times (n-k)$ zero matrix, and in the lower right corner a $(n-k)\times (n-k)$ matrix $\Delta_{\alpha}$ with $\det (\Delta_{\alpha}) \neq 0$. 
	\end{itemize}
	Consider 
	$$\mathcal{F}(\alpha)=\xi \left( \frac{1}{T} \int^{T}_{0} M_{{\bf z}_{\alpha}}(t) F_1(t,{\bf x}(t,{\bf z}_{\alpha})) \right)\;dt.$$
	If there exists $a_0 \in V$ with $\mathcal{F}(a_0)=0$ and $\det\left( \frac{d \mathcal{F}}{d \alpha} (a_0)\right)\neq 0$, then there is a $T$-periodic solution $\varphi(t,\varepsilon)$ of the system (\ref{aver}) such that $\varphi(t,\varepsilon) \rightarrow {\bf x}(t, {\bf z}_{a_0})$ as $\varepsilon \rightarrow 0$.    
\end{theorem}

\subsection{Descartes' Theorem} 

Now, we present the Descartes' Theorem, which is a fundamental tool to study the number of isolated periodic orbits bifurcating from a period annulus. For more details, see \cite{LMdescarte}.

\begin{theorem}[Descartes' Theorem]\label{descarte}
	Let $p(x)=a_{i_1}x^{i_1}+\dots+a_{i_r}x^{i_r}$ be a polynomial with real coefficients, $r>0$, $0\leq i_1 < \dots <i_r$ and $a_{i_j}$ are not simultaneously zero for $j\in \{1,2,\dots, r\}$. Let $s$  be the number of sign changes in the sequence $a_{i_1},\dots, a_{i_r}$, that is, the number of pairs of consecutive terms in the sequence that have opposite signs. Let $z$ be the number of positive real roots of $p(x)$ (counted with multiplicity). Then, $s-z$ is a non negative even number. Moreover, $p(x)$ has at most $r-1$ positive real roots. 
\end{theorem}

\begin{proof}
	Firstly, we are going to show that $s-z$ is even. Suppose, without loss of generality, that $a_{i_1}>0$, thus $p(0)\geq 0$. If $a_{i_r}>0$, we get $p(x)\rightarrow +\infty$ as $x\rightarrow +\infty$, thus the graph of the polynomial $p$ must cross the positive part of the $x$-axis an even number of times, that is, $z$ is even. Furthermore, $s$ is even, since the signs of $a_{i_1}$ and $a_{i_r}$ are positives. Then, $s-z$ is even. Now if $a_{i_r}<0$, we obtain  $p(x)\rightarrow -\infty$ as $x\rightarrow +\infty$, thus the graph of the polynomial $p$ must cross the positive part of the $x$-axis an odd number of times, that is, $z$ is odd. Moreover, $s$ is odd, since the signs of $a_{i_1}$ and $a_{i_r}$ are opposites. Hence, $s-z$ is even.
	
	Next, we are going to show that $s-z$ is non negative, that is, $s\geq z$. For this, we are going to use the principle of finite induction on $r$ and suppose $i_1=0$, without loss of generality. 
	
	If $r=2$, we can write $p(x)=ax^n+b$. If $a>0$, we have $p'(x)=nax^{n-1}>0$ for all $x>0$, hence, $p$ is strictly increasing for $x>0$. We know that $p(0)=b$ and, for all $x>0$, the polynomial $p$ is increasing. Then, if $b<0$, we get $s=1$ and for a certain $x>0$ the polynomial has value equal to zero, that is, there is an unique positive real root. Now, if $b>0$, we have $s=0$ and $p(x)>0$ for all $x>0$, that is, admits zero roots. For $a<0$, we have $p'(x)<0$ for all $x>0$, thus the same argument is valid changing the signs. 
	
	Suppose that all polynomials with $r=n-1$ satisfy $s\geq z$ and take a polynomial $p(x)=a_{i_1}x^{i_1}+\dots+a_{i_r}x^{i_n}$. Considering the polynomial $p'(x)$, we have that the number of sign changes $s'$ of $p'(x)$ is either $s$ or $s-1$, depending on whether there is a sign change at the last coefficient of $p(x)$. Moreover, the number of positive zeros $z'$ of $p'(x)$ is at least $z-1$, because between any two zeros of $p$ there is a zero of $p'$ by Rolle's Theorem. 
	Since $p'(x)$ has $n-1$ terms, the inductive hypothesis implies that $s'\geq z'$, hence
	$$s\geq s' \geq z'\geq z-1.$$
	Thus, $s-z \geq -1$, but since $s-z$ is even, we get $s-z \geq 0$.
	
	Finally, we have that $p(x)$ has $r$ non-zero terms, thus it admits at most $r-1$ sign changes. Therefore, since $s-z \geq 0$ implies $z \leq s$, we get that $p(x)$ has at most $r-1$ positive real roots. 
\end{proof}

\section{Proof of Theorem 1}\label{prteo1}

In this section, we show some results concerning the piecewise smooth differential system which admits the unit sphere as the discontinuity manifold. The main result presents conditions for the existence of pseudo-orbits. 

Consider $F(z,\dot{z},\ddot{z})=1$ and $b=\sgn(h(z,\dot{z},\ddot{z}))$, where $h(z,\dot{z},\ddot{z})=z^2+(\dot{z})^2+(\ddot{z})^2-1$. Then, we get $\dddot{z}+a\ddot{z}+b\dot{z}+abz=\varepsilon$, which is equivalent to the piecewise smooth differential system (\ref{sdsp}).

Initially, we are going to prove some lemmas in which the proof of Theorem \ref{teo1} is based.

\begin{lemma}
	Consider the system (\ref{sdsp}).
	\begin{itemize}
		\item[a)] The plane $\alpha_-=\left\lbrace (x,y,z) \in \mathbb{R}^3\; |\; -y+z=-\frac{\varepsilon}{a} \right\rbrace$ is invariant with respect to the vector field $X_-$, which behaves like a saddle in $\alpha_-$.
		\item[b)] The plane $\alpha_+=\left\lbrace (x,y,z) \in \mathbb{R}^3\; |\; y+z=\frac{\varepsilon}{a} \right\rbrace$ is invariant with respect to the vector field $X_+$, which behaves like a center in $\alpha_+$. 
	\end{itemize}
\end{lemma}

\begin{proof}
	a) We have that   
	$$X_-(x,y,z)=
	\left(\begin{array}{ccc}
		0 & 1 & 0\\
		1 & -a & a\\
		1 & 0 & 0
	\end{array} \right) \left(\begin{array}{c}
		x\\
		y\\
		z
	\end{array} \right)+\left(\begin{array}{c}
		0\\
		\varepsilon\\
		0
	\end{array} \right)$$
	admits $S_-=\left(0,0,-\frac{\varepsilon}{a} \right)$ as singularity. Furthermore, the characteristic polynomial of the matrix
	$$\left( \begin{array}{ccc}
		0 & 1 & 0\\
		1 & -a & a\\
		1 & 0 & 0
	\end{array} \right)$$
	is $p_{X_-}(\lambda)=-\lambda^3-a\lambda^2+\lambda+a$, thus the eigenvalues of $X_-$ are given by $\lambda_{1_-}=-a$, $\lambda_{2_-}=1$ and $\lambda_{3_-}=-1$.
	Determining the corresponding eigenvectors, we obtain 
		$v_{1_-}=(-a,a^2,1)$, $v_{2_-}=(1,1,1)$ and $v_{3_-}=(-1,1,1)$, relative to the eigenvalues $\lambda_{1_-}$, $\lambda_{2_-}$ and $\lambda_{3_-}$, respectively. 
		
		From these information, we can observe that the vector field $X_-$ behaves like a saddle in the plane spanned by the vectors $v_{2_-}$ and $v_{3-}$ with origin at the point $S_-$, that is, the vector $v_{2_-}$ starts at the point $S_-$ and ends at the point $\left(1,1,1-\frac{\varepsilon}{a}\right)$, while the vector $v_{3-}$ starts at the point $S_-$ and ends at the point $\left(-1,1,1-\frac{\varepsilon}{a}\right)$. Hence, we have 
		$$v_{2_-}\times v_{3_-}=\left| \begin{array}{ccc}
			i & j & k\\
			1 & 1 & 1\\
			-1 & 1 & 1
		\end{array}\right| = -2j+2k= (0,-2,2),$$
		thus, the plane satisfies
		$$\left(x,y,z+\frac{\varepsilon}{a}\right)\cdot (0,-2,2)=0 \Rightarrow -y+z=-\frac{\varepsilon}{a},$$
		that is,
		$$\alpha_-=\left\lbrace (x,y,z)\in \mathbb{R}^3|-y+z=-\frac{\varepsilon}{a}\right\rbrace .$$
		
		Finally, the solution of $X_-$ with initial condition $\left(x_0,y_0,-\frac{\varepsilon}{a}+y_0\right)$ is
        \begin{align}\label{solx-}
				x_{\alpha_-}(t)=&\left( x_0 \cosh (t)+y_0 \sinh (t),x_0 \sinh (t)+y_0 \cosh (t),\right.\\ \nonumber
            & \left.x_0 \sinh (t)+y_0 \cosh (t)-\frac{\varepsilon}{a}\right),
			\end{align}
		which belongs to $\alpha_-$, since
		$$-\frac{\varepsilon}{a}+x_0 \sinh (t)+y_0 \cosh (t)-\left( x_0 \sinh (t)+y_0 \cosh (t)\right) =-\frac{\varepsilon}{a}.$$
		Therefore, $\alpha_-$ is invariant with respect to the vector field $X_-$. 
		\\~\\
		b) We have that 
		$$X_+(x,y,z)=
		\left(\begin{array}{ccc}
			0 & 1 & 0\\
			-1 & -a & -a\\
			1 & 0 & 0
		\end{array} \right) \left(\begin{array}{c}
			x\\
			y\\
			z
		\end{array} \right)+\left(\begin{array}{c}
			0\\
			\varepsilon\\
			0
		\end{array} \right)$$
		admits $S_+=\left(0,0,\frac{\varepsilon}{a} \right)$ as singularity. Moreover, the characteristic polynomial of the matrix
		$$\left( \begin{array}{ccc}
			0 & 1 & 0\\
			-1 & -a & -a\\
			1 & 0 & 0
		\end{array} \right)$$ 
		is $p_{X_+}(\lambda)=-\lambda^3-a\lambda^2-\lambda-a$, thus the eigenvalues of $X_-$ are given by $\lambda_{1_+}=-a$, $\lambda_{2_+}=i$ and $\lambda_{3_+}=-i$. Determining the corresponding eigenvectors, we obtain  $v_{1_+}=(-a,a^2,1)$,  $v_{2_+}=(i,-1,1)=(0,-1,1)+i(1,0,0)$ and $v_{3_+}=(-i,-1,1)=(0,-1,1)+i(1,0,0)$, relative to the eigenvalues $\lambda_{1_+}$, $\lambda_{2_+}$ and $\lambda_{3_+}$, respectively.  
			
			From these information, we can observe that the vector field $X_+$ behaves like a center in the plane spanned by the vectors $w_{1_+}=(0,-1,1)$ and $w_{2_+}=(1,0,0)$ with origin at the point $S_+$, that is, the vector $w_{1_+}$ starts at the point $S_+$ and ends at the point $\left(0,-1,1+\frac{\varepsilon}{a}\right)$, while the vector $w_{2_+}$ starts at the point $S_+$ and ends at the point $\left(1,0,\frac{\varepsilon}{a}\right)$. Hence, we have 
			$$w_{1_+}\times w_{2_+}=\left| \begin{array}{ccc}
				i & j & k\\
				0 & -1 & 1\\
				1 & 0 & 0
			\end{array}\right| = j+k= (0,1,1),$$
			thus, the plane satisfies
			$$\left(x,y,z-\frac{\varepsilon}{a}\right)\cdot (0,1,1)=0 \Rightarrow y+z=\frac{\varepsilon}{a},$$
			that is, 
			$$\alpha_+=\left\lbrace (x,y,z) \in \mathbb{R}^3 | y+z=\frac{\varepsilon}{a} \right\rbrace .$$
			
			We note that the solution of $X_+$, with initial condition $\left(x_0,y_0,\frac{\varepsilon}{a}-y_0\right)$ is 
			\begin{align}\label{solx+}
				x_{\alpha_+}(t)=&\left( x_0 \cos (t)+y_0 \sin (t),y_0 \cos (t)-x_0 \sin (t),\right.\\ \nonumber
                &\left.\frac{\varepsilon}{a}+x_0 \sin (t)-y_0 \cos (t)\right),
			\end{align}
			which belongs to $\alpha_+$, since  
			$$\frac{\varepsilon}{a}+x_0 \sin (t)-y_0 \cos (t)+\left( y_0 \cos (t)-x_0 \sin (t)\right) =\frac{\varepsilon}{a}.$$
			Therefore, $\alpha_+$ is invariant with respect to the vector field $X_+$.
		\end{proof}
		
		\begin{lemma}\label{tangencia}
			The intersection points between the planes $\alpha_-$, $\alpha_+$ and the sphere $S^2$ are tangency points of $Z_{X_-X_+}$. 
		\end{lemma}
		
		\begin{proof}
			From the equations of the planes $\alpha_-$ and $\alpha_+$, we get that $z=0$ and $y=\frac{\varepsilon}{a}$. Substituting these values in the equation of the sphere $S^2$, that is, $x^2+y^2+z^2=1$,  we have
			$  x=\pm \frac{\sqrt{a^2-\varepsilon^2}}{a}.$
			Thus, we get two points,
			$$P=\left(- \frac{\sqrt{a^2-\varepsilon^2}}{a},\frac{\varepsilon}{a},0 \right) \quad \mbox{and} \quad Q=\left( \frac{\sqrt{a^2-\varepsilon^2}}{a},\frac{\varepsilon}{a},0 \right).$$
			
			We have that $h(x,y,z)=x^2+y^2+z^2-1$, then $\nabla h(x,y,z)=(2x, 2y, 2z)$, which implies
			$$X_-h(x,y,z)=  4xy-2ay^2+2ayz+2\varepsilon y+2xz$$
            and
            $$ X_+h(x,y,z) = -2ay^2-2ayz+2\varepsilon y+2xz.  $$
			Thus, 
			$$X_-h(P)=-\frac{4\varepsilon\sqrt{a^2-\varepsilon^2}}{a^2}, \quad X_-h(Q)=\frac{4\varepsilon\sqrt{a^2-\varepsilon^2}}{a^2},$$
            $$ X_+h(P)=0 \quad\mbox{and} \quad X_+h(Q)=0.$$
			Therefore, $P$ and $Q$ are tangency points.
		\end{proof}
		
		\begin{lemma}\label{tempos}
			Consider the solutions $x_{\alpha_+}(t)$ and $x_{\alpha_-}(t)$ with the point $P$ as initial condition. Then
			\begin{itemize}
				\item[a)] there is $t_-$ such that $x_{\alpha_-}(t_-)=Q$;
				\item[b)] there is $t_+$ such that $x_{\alpha_+}(t_+)=Q$.  
			\end{itemize}	
		\end{lemma}
		
		\begin{proof}
			a) By equation (\ref{solx-}), we know that the solution of the vector field $X_-$ with initial condition being the point $P$ is
			\begin{align*}
				x_{\alpha_-}(t)=&\left( \frac{\varepsilon \sinh (t)-\sqrt{a^2-\varepsilon^2} \cosh (t)}{a},\frac{\varepsilon \cosh (t)-\sqrt{a^2-\varepsilon^2} \sinh (t)}{a}\right. ,\\
				&\left. -\frac{\sqrt{a^2-\varepsilon^2} \sinh (t)-\varepsilon \cosh (t)+\varepsilon}{a}\right).
			\end{align*}
			
			We have that $x_{\alpha_-}(t_-)=Q$ implies
			\begin{align*}
                &\frac{\varepsilon \sinh (t_-)-     \sqrt{a^2-\varepsilon^2} \cosh (t_-)-\sqrt{a^2- \varepsilon^2}}{a}\\
                =&-\frac{\sqrt{a^2-    \varepsilon^2} \sinh (t_-)-\varepsilon \cosh (t_-)+\varepsilon}{a},
                \end{align*}
			which after some manipulations leads to
			$$t_-=\ln\left(\frac{\sqrt{a^2-\varepsilon^2}+\varepsilon}{\varepsilon-\sqrt{a^2-\varepsilon^2}} \right). $$
			\\~\\
			b) By equation (\ref{solx+}), we get that the solution of the vector field $X_+$ with initial condition being the point $P$ is
			\begin{align*} 
				x_{\alpha_+}(t)=&\left( \frac{\varepsilon \sin (t)-\sqrt{a^2-\varepsilon^2} \cos (t)}{a},\frac{\sqrt{a^2-\varepsilon^2} \sin (t)+\varepsilon \cos (t)}{a} ,\frac{-\sqrt{a^2-\varepsilon^2} \sin (t)-\varepsilon \cos (t)+\varepsilon}{a}\right). 
			\end{align*}
			
			We note that  $x_{\alpha_+}(t_+)=Q$ implies
			\begin{align*}
            \frac{\varepsilon \sin (t_+)-\sqrt{a^2-\varepsilon^2} \cos (t_+)-\sqrt{a^2-\varepsilon^2}}{a}=-\frac{-\sqrt{a^2-\varepsilon^2} \sin (t_+)-\varepsilon \cos (t_+)+\varepsilon}{a},
                \end{align*}
			which leads to
			$$\frac{\sin(t_+)+1}{\cos(t_+)}=\frac{\sqrt{a^2-\varepsilon^2}+\varepsilon}{\varepsilon-\sqrt{a^2-\varepsilon^2}}.$$
			Denoting $d=\dfrac{\sqrt{a^2-\varepsilon^2}+\varepsilon}{\varepsilon-\sqrt{a^2-\varepsilon^2}}$, we have that
			$$
			\dfrac{-1+d}{1+d}=\dfrac{-1+\dfrac{\sin(t_+)+1}{\cos(t_+)}}{-1+\dfrac{\sin(t_+)+1}{\cos(t_+)}}=\tan\left(\frac{t_+}{2} \right).
			$$
			Hence,
			$$t_+=2\arctan\left(\dfrac{\sqrt{a^2-\varepsilon^2}}{\varepsilon} \right). $$
		\end{proof}
		
		\begin{lemma}\label{sinal}
			Consider
			$$t_-=\ln \left(\dfrac{\sqrt{a^2-\varepsilon^2}+\varepsilon}{\varepsilon-\sqrt{a^2-\varepsilon^2}}\right)\quad \mbox{and} \quad t_+=2 \arctan\left(\dfrac{\sqrt{a^2-\varepsilon^2}}{\varepsilon}\right).$$
			Then 
			\begin{itemize}
				\item[a)] $t_+>0$ if, and only if, $\varepsilon>0$;
				\item[b)] $t_+<0$ if, and only if, $\varepsilon<0$;
				\item[c)] $t_->0$ if, and only if, $ \frac{|a|}{\sqrt{2}}<\varepsilon < |a|$;
				\item[d)] $t_-<0$ if, and only if, $-|a|< \varepsilon< -\frac{|a|}{\sqrt{2}}$. 
			\end{itemize}
		\end{lemma}
		
		\begin{proof}
			We initially remember that $t_+,\; t_- \in \mathbb{R}$, thus, 
			$$\sqrt{a^2-\varepsilon^2}>0 \Leftrightarrow a^2> \varepsilon^2 \Leftrightarrow |\varepsilon| < |a|.$$
			
			\noindent
			a) Since $t_+=2\arctan\left(\dfrac{\sqrt{a^2-\varepsilon^2}}{\varepsilon}\right)$, we get that  
			$t_+ >0$ if, and only if, $\dfrac{\sqrt{a^2-\varepsilon^2}}{\varepsilon}>0 $. As $\sqrt{a^2-\varepsilon^2}>0$, it follows that $t_+>0$ if, and only if, $\varepsilon>0$.
			
			\noindent
			b) We have that $t_+=2\arctan\left(\dfrac{\sqrt{a^2-\varepsilon^2}}{\varepsilon}\right)$, hence 
			$t_+ <0$ if, and only if, $\dfrac{\sqrt{a^2-\varepsilon^2}}{\varepsilon}>0 $, As $\sqrt{a^2-\varepsilon^2}>0$, it follows that $t_+<0$ if, and only if, $\varepsilon<0$.
			
			\noindent
			c) Since $t_-=\ln \left(\dfrac{\sqrt{a^2-\varepsilon^2}+\varepsilon}{\varepsilon-\sqrt{a^2-\varepsilon^2}}\right)$, we have that 
			$t_->0$ if, and only if, $\dfrac{\sqrt{a^2-\varepsilon^2}+\varepsilon}{\varepsilon-\sqrt{a^2-\varepsilon^2}}>1.$
			We observe that if $\varepsilon - \sqrt{a^2-\varepsilon^2}>0$, then $$t_->0 \Leftrightarrow \sqrt{a^2-\varepsilon^2}+\varepsilon > \varepsilon - \sqrt{a^2-\varepsilon^2} \Leftrightarrow \sqrt{a^2-\varepsilon^2}>0.$$ 
			Moreover, 
			$$\varepsilon> \sqrt{a^2-\varepsilon^2}>0 \Leftrightarrow \varepsilon^2> a^2-\varepsilon^2 \Leftrightarrow \varepsilon>\frac{|a|}{\sqrt{2}}.$$
			Now, if $\varepsilon - \sqrt{a^2-\varepsilon^2}<0$, then $$t_->0 \Leftrightarrow \sqrt{a^2-\varepsilon^2}+\varepsilon < \varepsilon - \sqrt{a^2-\varepsilon^2} \Leftrightarrow \sqrt{a^2-\varepsilon^2}<0,$$
			which is an absurd since $\sqrt{a^2-\varepsilon}>0$. Hence, $t_->0$, if and only if, $ \frac{|a|}{\sqrt{2}}<\varepsilon < |a|$, because $\varepsilon>\frac{|a|}{\sqrt{2}}$ and $|\varepsilon|<|a| \Rightarrow \varepsilon<|a|$.  
			
			\noindent
			d) We have that 
			$t_-=\ln \left(\dfrac{\sqrt{a^2-\varepsilon^2}+\varepsilon}{\varepsilon-\sqrt{a^2-\varepsilon^2}}\right)$, thus, $t_-<0$ if, and only if,
			$$0<\frac{\sqrt{a^2-\varepsilon^2}+\varepsilon}{\varepsilon-\sqrt{a^2-\varepsilon^2}}<1.$$
			Initially, we focus in the second part of this inequality. If $\varepsilon -\sqrt{a^2-\varepsilon^2}>0$, then 
			$$\frac{\sqrt{a^2-\varepsilon^2}+\varepsilon}{\varepsilon-\sqrt{a^2-\varepsilon^2}}<1 \Leftrightarrow \sqrt{a^2-\varepsilon^2}+\varepsilon < \varepsilon-\sqrt{a^2-\varepsilon^2} \Leftrightarrow \sqrt{a^2-\varepsilon^2}<0,$$
			which is an absurd, because $\sqrt{a^2-\varepsilon^2}>0$. Now, if $\varepsilon - \sqrt{a^2-\varepsilon^2}<0$, then
			$$\frac{\sqrt{a^2-\varepsilon^2}+\varepsilon}{\varepsilon-\sqrt{a^2-\varepsilon^2}}<1 \Leftrightarrow \sqrt{a^2-\varepsilon^2}+\varepsilon > \varepsilon-\sqrt{a^2-\varepsilon^2} \Leftrightarrow \sqrt{a^2-\varepsilon^2}>0.$$
			Moreover, we remember that by the first part of the inequality, that is, 
			$$0<\frac{\sqrt{a^2-\varepsilon^2}+\varepsilon}{\varepsilon-\sqrt{a^2-\varepsilon^2}},$$ we have that 
			\begin{align*}
                \sqrt{a^2-\varepsilon^2}+\varepsilon<0 \Leftrightarrow\; & 0<\sqrt{a^2-\varepsilon^2}<-\varepsilon \Leftrightarrow a^2-\varepsilon^2<(-\varepsilon)^2\\ \Leftrightarrow\; & a^2<2\varepsilon^2  \Leftrightarrow -\varepsilon>\frac{|a|}{\sqrt{2}}.
                \end{align*}
			Thus, $t_-<0$ if, and only if, $-|a|<\varepsilon < -\frac{|a|}{\sqrt{2}}$, since $\varepsilon<-\frac{|a|}{\sqrt{2}}$ and $|\varepsilon|<|a| \Rightarrow \varepsilon>-|a|$. 
		\end{proof}	
		
		\begin{lemma}\label{norma}
			Consider 
			$$ t_+=2 \arctan\left(\dfrac{\sqrt{a^2-\varepsilon^2}}{\varepsilon}\right).$$ 
			Suppose that $\frac{|a|}{\sqrt{2}}<\varepsilon <|a|$ or  $-|a|< \varepsilon < -\frac{|a|}{\sqrt{2}}$.  Then $\|x_{\alpha_+}(t)\|^2>1$ for $0<t<t_+$.
		\end{lemma}	
		
		\begin{proof}
			We are going to analyze the squared norm of the solution $x_{\alpha_+}(t)$, that is, 
			$$\begin{array}{cll}
				\vspace{0.5em}\|x_{\alpha_+}(t)\|^2  &=& \displaystyle \frac{\left(\varepsilon \sin (t)-\sqrt{a^2-\varepsilon^2} \cos (t)\right) ^2+\left( \sqrt{a^2-\varepsilon^2} \sin (t)+\varepsilon \cos (t)\right)^2}{a^2}\\
				&& +\,\displaystyle\frac{\left( -\sqrt{a^2-\varepsilon^2} \sin (t)-\varepsilon \cos (t)+\varepsilon\right)^2}{a^2}.
			\end{array}$$
			
			The intersection between the plane $\alpha_+$, parameterized by $z=\frac{\varepsilon}{a}-y$, and the sphere $x^2+y^2+z^2=1$ is given by the curve
			$$\dfrac{x^2}{\dfrac{2a^2-\varepsilon^2}{2a^2}}+\dfrac{\left(y-\dfrac{\varepsilon}{2a}\right)^2}{\dfrac{2a^2-\varepsilon^2}{4a^2}}=1.$$
			Considering the parametrization
			$$x=\frac{\sqrt{2a^2-\varepsilon^2}}{\sqrt{2}a}\cos(t),$$
			$$y-\frac{\varepsilon}{2a}=\sqrt{\frac{2a^2-\varepsilon^2}{4a^2}}\sin(t)\Rightarrow y=\frac{\varepsilon}{2a}+\frac{\sqrt{2a^2-\varepsilon^2}}{4a^2}\sin(t),$$
			$$z=\frac{\varepsilon}{a}-\frac{\varepsilon}{2a}-\frac{\sqrt{2a^2-\varepsilon^2}}{4a^2}\sin(t)=\frac{\varepsilon}{2a}-\frac{\sqrt{2a^2-\varepsilon^2}}{4a^2}\sin(t),$$
			we obtain an ellipse, parameterized by
			$$\mathcal{E}_{\alpha_+}(t)=\left(\frac{\sqrt{2a^2-\varepsilon^2}}{\sqrt{2}a}\cos(t), \frac{\varepsilon}{2a}+\frac{\sqrt{2a^2-\varepsilon^2}}{2a}\sin(t), \frac{\varepsilon
			}{2a}-\frac{\sqrt{2a^2-\varepsilon^2}}{2a}\sin(t) \right).$$
			
			Now, we are going to determine the intersection between $\mathcal{E}_{\alpha_+}(t)$ and $x_{\alpha_+}(s)$.
			This intersection is given by the solution of the system
			$$
			\def\arraystretch{2}\left\{\begin{array}{ccc}
				2 \sqrt{a^2-\varepsilon^2} \cos (s)+\sqrt{4 a^2-2 \varepsilon^2} \cos (t)-2 \varepsilon \sin (s)&=&0,\\
				-2 \sqrt{a^2-\varepsilon^2} \sin (s)+\sqrt{2 a^2-\varepsilon^2} \sin (t)-2 \varepsilon \cos (s)+\varepsilon&=&0.
			\end{array}\right.
			$$
			From the second equation, we have
			$$\sin (t)= \frac{2 \sqrt{a^2-\varepsilon ^2} \sin (s)+2 \varepsilon  \cos (s)-\varepsilon }{\sqrt{2 a^2-\varepsilon ^2}},$$
			if
			\begin{equation}\label{des1}
				-1\leq\frac{2 \sqrt{a^2-\varepsilon ^2} \sin (s)+2 \varepsilon  \cos (s)-\varepsilon }{\sqrt{2 a^2-\varepsilon ^2}}\leq 1.
			\end{equation}
			So, considering this hypothesis, we have
                \begin{align*}
			\cos\left(\arcsin\left(\frac{2 \sqrt{a^2-\varepsilon ^2} \sin (s)+2 \varepsilon  \cos (s)-\varepsilon }{\sqrt{2 a^2-\varepsilon ^2}}\right)\right)=\sqrt{1-\frac{\left(2 \sqrt{a^2-\varepsilon ^2} \sin (s)+2 \varepsilon  \cos (s)-\varepsilon \right)^2}{2 a^2-\varepsilon ^2}}.
                \end{align*}
			Substituting in the first equation, 
			\begin{align*}
                \sqrt{4 a^2-2 \varepsilon ^2} \sqrt{1-\frac{\left(2 \sqrt{a^2-\varepsilon ^2} \sin (s)+2 \varepsilon  \cos (s)-\varepsilon \right)^2}{2 a^2-\varepsilon ^2}}+2 \sqrt{a^2-\varepsilon ^2} \cos (s)-2 \varepsilon  \sin (s)=0,\end{align*}
			we obtain that 
			$$\sin \left(\frac{s}{2}\right) \left(-2 \varepsilon  \sqrt{a^2-\varepsilon ^2} \sin (s)+\left(a^2-2 \varepsilon ^2\right) \cos (s)+a^2\right)=0.$$
			Hence,
			$$ \sin \left(\frac{s}{2}\right)=0 \quad \mbox{or} \quad -2 \varepsilon  \sqrt{a^2-\varepsilon ^2} \sin (s)+\left(a^2-2 \varepsilon ^2\right) \cos (s)+a^2=0.$$
			We have that $\sin \left(\frac{s}{2}\right)=0$ when $s=4\pi c_1$, where $c_1 \in \mathbb{Z}$. Moreover,
			$$-2 \varepsilon  \sqrt{a^2-\varepsilon ^2} \sin (s)+\left(a^2-2 \varepsilon ^2\right) \cos (s)+a^2=0$$
			implies	
			$$\cos(s)=\frac{2 \varepsilon ^2-a^2}{a^2},$$
			so that 
			$s= \arccos\left(\frac{2 \varepsilon ^2-a^2}{a^2}\right)+2 \pi  c_2$, where $c_2\in \mathbb{Z}$.
			Since $x_{\alpha_+}$ is $2\pi$-periodic, we can consider only $s_1=0$ and $s_2=\arccos\left(\frac{2 \varepsilon ^2-a^2}{a^2}\right)$ the times we have the intersection between $\mathcal{E}_{\alpha_+}$ and $x_{\alpha_+}$, where
			$x_{\alpha_+}(s_1)=P$ and $x_{\alpha_+}(s_2)=Q$. Furthermore, 
		  \begin{align*}
            &\frac{2 \sqrt{a^2-\varepsilon ^2} \sin (s_1)+2 \varepsilon  \cos (s_1)-\varepsilon }{\sqrt{2 a^2-\varepsilon ^2}}\\
            =&\frac{2 \sqrt{a^2-\varepsilon ^2} \sin (s_2)+2 \varepsilon  \cos (s_2)-\varepsilon }{\sqrt{2 a^2-\varepsilon ^2}}\\
            =&\frac{\varepsilon }{\sqrt{2 a^2-\varepsilon ^2}}
            \end{align*}
			satisfies the inequality (\ref{des1}) for $0<\varepsilon\leq |a|$ or  $-|a|\leq \varepsilon<0$.
			
			Now, remembering that 
			$$t_+=2 \arctan\left(\dfrac{\sqrt{a^2-\varepsilon^2}}{\varepsilon}\right)\Rightarrow \tan\left(\frac{t_+}{2}\right)=\dfrac{\sqrt{a^2-\varepsilon^2}}{\varepsilon},$$
			and considering $\frac{t_+}{2}$ as an angle of a right triangle, we get
			$$\cos\left(\frac{t_+}{2}\right)=\dfrac{\varepsilon}{a}\quad \mbox{and} \quad \sin\left(\frac{t_+}{2}\right)=\dfrac{\sqrt{a^2-\varepsilon^2}}{a}.$$ 
			Using trigonometric identities, we obtain
			$$\cos(t_+)=\cos^2\left(\frac{t_+}{2}\right)-\sin^2\left(\frac{t_+}{2}\right)=\dfrac{2\varepsilon^2-a^2}{a^2}.$$
			Hence, $t_+=s_2.$
			
			Finally, consider $$s_m=\frac{\arccos\left(\dfrac{2 \varepsilon ^2-a^2}{a^2}\right)}{2}.$$
			Substituting it in the expression of the norm, we have
			$$\|x_{\alpha_+}(s_m)\|^2=\frac{\varepsilon ^2}{a^2}-\frac{2 \varepsilon }{a}+2>1 \Leftrightarrow (\varepsilon-a)^2>0 \Leftrightarrow \varepsilon>a \quad \mbox{or} \quad \varepsilon<a.$$
			
			Therefore, since $x_{\alpha_+}(t)$ intersects $\mathcal{E}_{\alpha_+}$ in two points, $x_{\alpha_+}$ is continuous and $\|x_{\alpha_+}(s_m)\|^2>1$, for $\frac{|a|}{\sqrt{2}}<\varepsilon <|a|$ or  $-|a|< \varepsilon < -\frac{|a|}{\sqrt{2}}$, we have $\|x_{\alpha_+}(t)\|^2>1$ for $0<t<t_+$. 	
		\end{proof}

		\begin{lemma}\label{norma2}
			Consider 
			$$t_-=\ln \left(\dfrac{\sqrt{a^2-\varepsilon^2}+\varepsilon}{\varepsilon-\sqrt{a^2-\varepsilon^2}}\right).$$ 
			Suppose that $\frac{|a|}{\sqrt{2}}<\varepsilon <|a|$ or  $-|a|< \varepsilon < -\frac{|a|}{\sqrt{2}}$. Then $\|x_{\alpha_-}(t)\|^2<1$ for $0<t<t_-$.
		\end{lemma}	
		
		\begin{proof}
			We will analyze the squared norm of the solution $x_{\alpha_-}(t)$, that is,
			\[\begin{array}{lcl}
				\vspace{0.5em}\displaystyle \|x_{\alpha_-}(t)\|^2  &= & \displaystyle \frac{\left( \varepsilon \sinh (t)-\sqrt{a^2-\varepsilon^2} \cosh (t)\right)^2 }{a^2}+\, \displaystyle \frac{\left( \varepsilon \cosh (t)-\sqrt{a^2-\varepsilon^2} \sinh (t)\right)^2}{a^2}\\
				& & + \,\displaystyle \frac{\left( \sqrt{a^2-\varepsilon^2} \sinh (t)-\varepsilon \cosh (t)+\varepsilon\right)^2}{a^2}.     
			\end{array}\]
			
			Considering the same idea developed in the proof of Lemma \ref{norma}, we obtain that the intersection between the plane $\alpha_-$ and the sphere $S^2$, is the ellipse parameterized by
			$$\mathcal{E}_{\alpha_-}(t)\!\!=\!\!\left(\!\frac{\sqrt{2a^2-\varepsilon^2}}{\sqrt{2}a}\cos(t), \frac{\varepsilon}{2a}+\frac{\sqrt{2a^2-\varepsilon^2}}{2a}\sin(t), -\frac{\varepsilon
			}{2a}+\frac{\sqrt{2a^2-\varepsilon^2}}{2a}\sin(t) \!\!\right)\!.$$
			
			Now, we are going to study the intersection between $\mathcal{E}_{\alpha_-}(t)$ and $x_{\alpha_-}(s)$. Equaling them, we obtain the system
			$$
			\def\arraystretch{2}\left\{\begin{array}{ccc}
				2 \sqrt{a^2-\varepsilon^2} \cosh (s)+\sqrt{4 a^2-2 \varepsilon^2} \cos (t)-2 \varepsilon \sinh (s) & = & 0,\\
				2 \sqrt{a^2-\varepsilon^2} \sinh (s)+\sqrt{2 a^2-\varepsilon^2} \sin (t)-2 \varepsilon \cosh (s)+\varepsilon&=&0.
			\end{array}\right.
			$$
			From the second equation, we get 
			$$\sin (t)=\frac{-2 \sqrt{a^2-\varepsilon^2} \sinh (s)+2 \varepsilon \cosh (s)-\varepsilon^2}{\sqrt{2 a^2-\varepsilon^2}},$$
			if
			\begin{equation}\label{des2}
				-1\leq\frac{-2 \sqrt{a^2-\varepsilon^2} \sinh (s)+2 \varepsilon \cosh (s)-\varepsilon}{\sqrt{2 a^2-\varepsilon^2}} \leq 1.
			\end{equation}	
			So, considering this hypothesis, we have
			\begin{align*}
			    \cos\left(\arcsin\left(\frac{2 \varepsilon \cosh (s)-2 \sqrt{a^2-\varepsilon^2} \sinh (s)-\varepsilon}{\sqrt{2 a^2-\varepsilon^2}}\right)\right)=\sqrt{1-\frac{\left(-2 \sqrt{a^2-\varepsilon^2} \sinh (s)+2 \varepsilon \cosh (s)-\varepsilon\right)^2}{2 a^2-\varepsilon^2}}.
                \end{align*}
			Substituting in the first equation,
			\begin{align*}
                \sqrt{4 a^2-2 \varepsilon^2}\sqrt{1-\frac{\left(2 \varepsilon \cosh(s)-2 \sqrt{a^2-\varepsilon^2} \sinh (s)-       \varepsilon\right)^2}{2 a^2-\varepsilon^2}}= 2\varepsilon \sinh (s)-2 \sqrt{a^2-\varepsilon^2} \cosh (s),
                \end{align*}
			we get $s_1=0$ and 
			\begin{align*}
				s_2 =&\displaystyle \ln \left[ \left( 2 \varepsilon \sqrt{a^2-\varepsilon^2}-3 a^2 +4 \varepsilon^2 \right.\right. +\displaystyle 2 \left(27 a^6-286 \varepsilon^5 \sqrt{a^2-\varepsilon^2}+2 \varepsilon^6\right.\\
				&\left. +a^2 \varepsilon^3 \left(253 \sqrt{a^2-\varepsilon^2}+141 \varepsilon\right)-18 a^4 \varepsilon \left(3 \sqrt{a^2-\varepsilon^2}+7 \varepsilon\right)+9 \sqrt{3}\right.\\
				& +\!\!\displaystyle \left. \sqrt{\!-\!\left(\left(17 \varepsilon^2-9 a^2\right) \left(a^2 \varepsilon-2 \varepsilon^3\right)^2 \left(-a^4+4 a^2 \varepsilon \left(\sqrt{a^2-\varepsilon^2}-d\right)+4 \varepsilon^4\!\right)\!\right)\!}\!\right)^{\frac{1}{3}}\\
				& +\displaystyle \left(18 a^4-2 a^2 \varepsilon \left(12 \sqrt{a^2-\varepsilon^2}+23 \varepsilon\right)+4 \varepsilon^3 \left(11 \sqrt{a^2-\varepsilon^2}+6 \varepsilon\right)\right)\\
				&\div\displaystyle \left(27 a^6-286 \varepsilon^5 \sqrt{a^2-\varepsilon^2}+a^2 \varepsilon^3 \left(253 \sqrt{a^2-\varepsilon^2}+141 \varepsilon\right)+2 \varepsilon^6\right.\\
				&+\!\displaystyle \left.\left.9 \sqrt{3} \sqrt{\!-\varepsilon^2 \left(9 a^2-17 \varepsilon^2\right) \left(a^2-2 \varepsilon^2\right)^2 \left(a^4+4 a^2 \varepsilon \left(\varepsilon-\sqrt{a^2-\varepsilon^2}\right)-4 \varepsilon^4\right)\!}\right.\right.\\
				&\left.\left.-18 a^4 \varepsilon \left(3 \sqrt{a^2-\varepsilon^2}+7 \varepsilon\right)\right)^{\frac{1}{3}}\right) \div \displaystyle \left. 9 \left(a^2-2 \varepsilon \sqrt{a^2-\varepsilon^2}\right)\right].
			\end{align*}
			It is possible to notice that $t_-=s_2$. Moreover, 
		\begin{align*}
        &\frac{2 \varepsilon \cosh (s_1)-2 \sqrt{a^2-\varepsilon^2} \sinh (s_1)-\varepsilon}{\sqrt{2 a^2-\varepsilon^2}}\\
        =&\frac{2 \varepsilon \cosh (s_2)-2 \sqrt{a^2-\varepsilon^2} \sinh (s_2)-\varepsilon}{\sqrt{2 a^2-\varepsilon^2}}\\
        =&\frac{\varepsilon}{\sqrt{2 a^2-\varepsilon^2}}
        \end{align*}
			satisfies the inequality (\ref{des2}) for $0<\varepsilon\leq |a|$ or  $-|a|\leq \varepsilon<0$.
			
			Finally, consider
			$$s_m=\frac{1}{2}\ln\left(\dfrac{\sqrt{a^2-\varepsilon^2}+\varepsilon}{\varepsilon-\sqrt{a^2-\varepsilon^2}}\right).$$
			Substituting it in the expression of the norm, we get
			$$\|x_{\alpha_-}(s_m)\|^2=\frac{\varepsilon \left(2 \sqrt{a^2-\varepsilon^2} \sqrt{\frac{\sqrt{a^2-\varepsilon^2}+\varepsilon}{\varepsilon-\sqrt{a^2-\varepsilon^2}}}+\varepsilon \left(5-2 \sqrt{\frac{\sqrt{a^2-\varepsilon^2}+\varepsilon}{\varepsilon-\sqrt{a^2-\varepsilon^2}}}\right)\right)}{a^2}-2<1$$
			which implies
			$$ a<0 \quad \mbox{and}\quad \left(a< \varepsilon\leq \frac{a}{\sqrt{2}} \quad \mbox{or} \quad -\frac{a}{\sqrt{2}}< \varepsilon<-a\right)$$
			or
			$$ a>0\quad \mbox{and} \quad \left(-a< \varepsilon\leq -\frac{a}{\sqrt{2}}\quad \mbox{or}\quad \frac{a}{\sqrt{2}}< \varepsilon<a\right).$$
			
			Therefore, since $x_{\alpha_-}(t)$ intersects $\mathcal{E}_{\alpha_-}$ in two points, $x_{\alpha_-}$ is continuous and $\|x_{\alpha_-}(s_m)\|^2<1$, for $\frac{|a|}{\sqrt{2}}<\varepsilon <|a|$ or  $-|a|< \varepsilon < -\frac{|a|}{\sqrt{2}}$, we have $\|x_{\alpha_-}(t)\|^2<1$ for $0<t<t_-$. 
		\end{proof}
		
		From the previous lemmas, we are able to present a proof for Theorem \ref{teo1}, as follows.

		\begin{proof}[Proof of Theorem 1]
			From Lemma \ref{tangencia}, we obtain the tangency points $P$ and $Q$, which belongs to the solutions $x_{\alpha_+}(t)$ and $x_{\alpha_-}(t)$ with initial condition $P$. Moreover, by Lemma \ref{tempos}, we get the times $t_+$ and $t_-$, in which $x_{\alpha_+}(t)$ and $x_{\alpha_-}(t)$ pass through the point $Q$, respectively. Even more, by Lemma \ref{sinal}, if $\frac{|a|}{\sqrt{2}}<\varepsilon<|a|$, we have $t_->0$ and $t_+>0$ and if $-|a|<\varepsilon<-\frac{|a|}{\sqrt{2}}$, we have $t_-<0$ and $t_+<0$. 
			Finally, by Lemma \ref{norma}, we obtain that the solution $x_{\alpha_+}(t)$ restricted to the points $P$ and $Q$ is visible and, by Lemma \ref{norma2}, we get that the solution $x_{\alpha_-}(t)$ restricted to the points $P$ and $Q$ is visible. This is valid considering the hypotheses of the lemmas. Therefore, the third order differential equation admits a pseudo-cycle when 
			$$\frac{|a|}{\sqrt{2}}<\varepsilon<|a| \quad \mbox{or}\quad -|a|<\varepsilon<-\frac{|a|}{\sqrt{2}}.$$
		\end{proof}
		
		In Figure \ref{figpseudo}, we illustrate the pseudo-cycle of the piecewise smooth differential system $Z_{X_-X_+}$ with $a=5$ and $\varepsilon=4$, that is, $X_-(x,y,z)=(y,-5y+x+5z+4,x)$ and $X_+(x,y,z)=(y,-5y-x-5z+4,x)$.
		
		\begin{figure}[h]
			\begin{subfigure}{0.33\textwidth}
				\centering
				\includegraphics[scale=0.4]{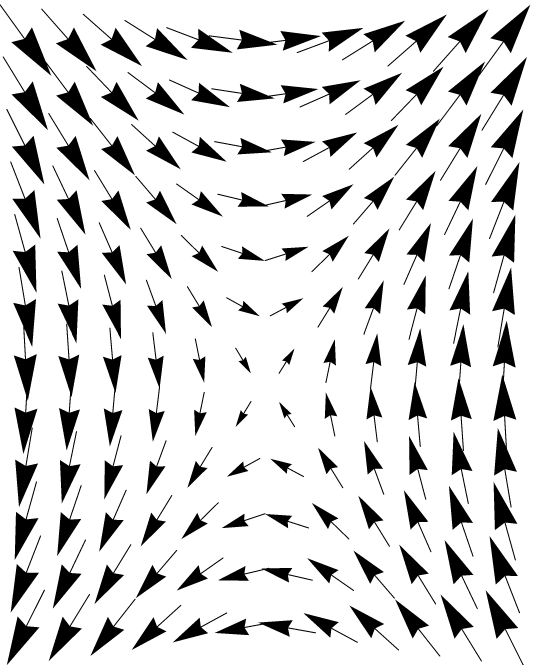}
				\caption{$X_-$ in the plane $\alpha_-$}
				\label{sela3D}
			\end{subfigure}%
			\begin{subfigure}{0.33\textwidth}
				\centering
				\includegraphics[scale=0.45]{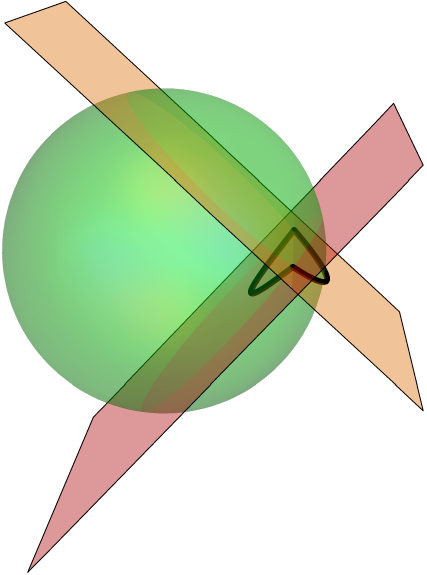}
				\caption{Pseudo-trajectory}
			\end{subfigure}%
			\begin{subfigure}{0.33\textwidth}
				\centering
				\includegraphics[scale=0.4]{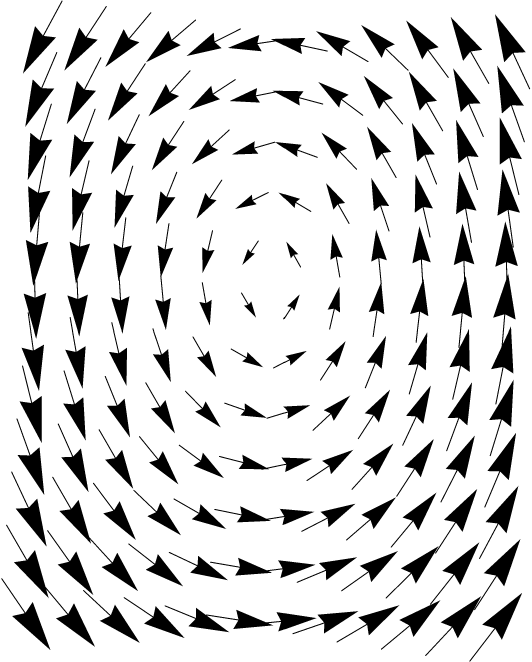}
				\caption{$X_+$ in the plane $\alpha_+$}
				\label{centro3D}
			\end{subfigure}%
			\caption{Vector field $Z_{X_-X_+}$ with $a=5$ and $\varepsilon=4$. The heart-shaped curve in (b) is the pseudo-cycle.}
			\label{figpseudo}
		\end{figure} 
		
		\section{Proof of Theorem 2}\label{prteo2}

        In this section, we study the perturbation of a smooth vector field. The averaging theory is employed to show a result concerning the number of limit cycles.
        
		We are considering $\varepsilon \neq 0$ sufficiently small, $b>0$ and $F(z,\dot{z},\ddot{z})$ a polynomial of degree $n$ in the differential equation (\ref{maineq}).
		
		Here, we also present some lemmas that will support us in the proof of Theorem \ref{teo2}, which will be given in two parts.

		\begin{proof}[Proof of Theorem \ref{teo2} for $\varepsilon = 0$]
			Analyzing the unperturbed system
			$$\left(\begin{array}{c}
				\dot{x}\\
				\dot{y}\\
				\dot{z}
			\end{array} \right)=
			\left(\begin{array}{ccc}
				0 & 1 & 0\\
				-b & -a & -ab\\
				1 & 0 & 0
			\end{array} \right) \left(\begin{array}{c}
				x\\
				y\\
				z
			\end{array} \right),$$
			that admits $S=(0,0,0)$ as singularity, we obtain that the characteristic polynomial of the matrix 
			$$\left(\begin{array}{ccc}
				0 & 1 & 0\\
				-b & -a & -ab\\
				1 & 0 & 0
			\end{array} \right)$$
			is $p(\lambda)=-\lambda^3-a\lambda^2-b\lambda-ab$. Thus the eigenvalues are given by $\lambda_1=-a$, $\lambda_2=\sqrt{b} i$ and $\lambda_3=-\sqrt{b} i$.
			Determining the corresponding eigenvectors, we obtain $u_1=(-a,a^2,1)$, $u_2=(\sqrt{b}i,-b,1)=(0,-b,1)+i(\sqrt{b},0,0)$ and $u_3=(-\sqrt{b}i,-b,1)=(0,-b,1)-i(\sqrt{b},0,0)$, relative to the eigenvalues $\lambda_1$, $\lambda_2$ and $\lambda_3$, respectively.  
			
			From these information, we note that the system $(\ref{sisper})|_{\varepsilon=0}$ behaves like a center in the plane spanned by the vectors $w_1=(0,-b,1)$ and $w_2=(\sqrt{b},0,0)$ with origin at the point $\left(0,0,0\right)$. Hence, we have 
			$$w_1\times w_2=\left| \begin{array}{ccc}
				i & j & k\\
				0 & -b & 1\\
				\sqrt{b} & 0 & 0
			\end{array}\right| = \sqrt{b}j+b\sqrt{b}k= (0,\sqrt{b},b\sqrt{b}),$$
			thus, the plane is given by
			$$ \left(x,y,z\right)\cdot (0,\sqrt{b},b\sqrt{b})=0\Rightarrow y+bz=0, $$
			that is, 
			$$ \alpha=\left\lbrace (x,y,z) \in \mathbb{R}^3\; |\; y+bz=0 \right\rbrace .$$
			
			We note that the solution of the system $(\ref{sisper})|_{\varepsilon=0}$ with initial condition $\left(x_0,-bz_0,z_0\right)$ is
		  \begin{align*}
            x_{\alpha}(t)=&\left( -z_0\sqrt{b} \sin(\sqrt{b}t)+x_0\cos(\sqrt{b}t),-z_0b\cos(\sqrt{b}t)-x_0\sqrt{b}\sin(\sqrt{b}t),\right.\\
            &\left.z_0\cos (\sqrt{b}t)+\frac{x_0}{\sqrt{b}}\sin(\sqrt{b}t)\right),
            \end{align*}
			which belongs to $\alpha$, because
			$$-z_0b\cos(\sqrt{b}t)-x_0\sqrt{b}\sin(\sqrt{b}t)+b\left(z_0\cos (\sqrt{b}t)+\frac{x_0}{\sqrt{b}}\sin(\sqrt{b}t) \right)  =0.$$
			Therefore, $\alpha$ is invariant with respect to the unperturbed system. Moreover, the orbits admit period $\frac{2\pi}{\sqrt{b}}$.
		\end{proof}

		\begin{lemma}\label{bifurcate}
			Assume $b>0$. For $\varepsilon\neq 0$ sufficiently small and for each simply root $r_0$ of the polynomial
		\begin{align*}
	&\int_0^{\frac{2\pi}{\sqrt{b}}} \left(\sqrt{b} \cos         \left(\sqrt{b} \theta \right)-a \sin \left(\sqrt{b}         \theta \right)\right)\\
        &F\left(\sqrt{b} r \sin \left(\sqrt{b} \theta \right),-b r \cos\left(\sqrt{b} \theta \right),r \cos \left(\sqrt{b}\theta\right)\right) \, d\theta,
        \end{align*}
			the differential system (\ref{sisper}) admits a limit cycle bifurcating from a periodic orbit of the plane $\alpha$ of the unperturbed system $(\ref{sisper})|_{\varepsilon=0}$.  
		\end{lemma}
		
		\begin{proof}
			Consider 
			$$B=\left(
			\begin{array}{ccc}
				0 & \sqrt{b} & -a \\
				-b & 0 & a^2 \\
				1 & 0 & 1 \\
			\end{array}
			\right) \quad \mbox{and} \quad J=\left(
			\begin{array}{ccc}
				0 & \sqrt{b} & 0 \\
				-\sqrt{b} & 0 & 0 \\
				0 & 0 & -a \\
			\end{array}
			\right)$$
			the matrix of the eigenvectors and the Jordan matrix of
			$$
			A=\left(\begin{array}{ccc}
				0 & 1 & 0\\
				-b & -a & -ab\\
				1 & 0 & 0
			\end{array} \right),
			$$
			respectively. We have that 
			$$B^{-1}=
			\left(
			\def\arraystretch{2.5}\begin{array}{ccc}
				0 & -\dfrac{\sqrt{b}}{a^2 \sqrt{b}+b^{3/2}} & \dfrac{a^2 \sqrt{b}}{a^2 \sqrt{b}+b^{3/2}} \\
				\dfrac{a^2+b}{a^2 \sqrt{b}+b^{3/2}} & \dfrac{a}{a^2 \sqrt{b}+b^{3/2}} & \dfrac{a b}{a^2 \sqrt{b}+b^{3/2}} \\
				0 & \dfrac{\sqrt{b}}{a^2 \sqrt{b}+b^{3/2}} & \dfrac{b^{3/2}}{a^2 \sqrt{b}+b^{3/2}} \\
			\end{array}
			\right)
			$$
			and $BJB^{-1}=A$. Consider the change of coordinates $$(X,Y,Z)^{T}=B^{-1}(x,y,z)^{T},$$ that is, 
			$x = \sqrt{b} Y-a Z,\; y =a^2 Z-b X,\; z =X + Z$. Thus, we get
			\begin{equation}
				\left\{\def\arraystretch{3}\begin{array}{lll}
					\dot{X}&=&\sqrt{b} Y-\dfrac{\varepsilon  F\left(\sqrt{b} Y-a Z,a^2 Z-b X,X+Z\right)}{a^2+b},\\
					\dot{Y}&=&-\sqrt{b} X+\dfrac{a \varepsilon  F\left(\sqrt{b} Y-a Z,a^2 Z-b X,X+Z\right)}{\sqrt{b}(a^2+b)},\\
					\dot{Z}&=&-a Z+\dfrac{\varepsilon  F\left(\sqrt{b} Y-a Z,a^2 Z-b X,X+Z\right)}{a^2+b}. 
				\end{array}\right.
			\end{equation}
			
			Next, we are going to consider the change of cylindrical coordinates, that is, $X=r\cos(\sqrt{b}\theta),\; Y=r\sin(\sqrt{b}\theta),\; Z=Z$. Thus, we obtain 
			\begin{equation}\label{spercil}\left\{\begin{array}{lll}
					\dot{r}&=& \dfrac{\varepsilon  \left(a \sin \left(\sqrt{b} \theta \right)-\sqrt{b} \cos \left(\sqrt{b} \theta \right)\right) }{\sqrt{b} \left(a^2+b\right)}\\[0.3cm] 
					& & F\left(\sqrt{b} r \sin \left(\sqrt{b} \theta \right)-a Z,a^2 Z-b r \cos \left(\sqrt{b} \theta \right),r \cos \left(\sqrt{b} \theta \right)+Z\right),\\[0.3cm] 
					\dot{\theta}&=&-1+\dfrac{\varepsilon  \left(a \cos \left(\sqrt{b} \theta \right)+\sqrt{b} \sin \left(\sqrt{b} \theta \right)\right) }{b r \left(a^2+b\right)}\\[0.3cm] 
					& &  F\left(\sqrt{b} r \sin \left(\sqrt{b} \theta \right)-a Z,a^2 Z-b r \cos \left(\sqrt{b} \theta \right),r \cos \left(\sqrt{b} \theta \right)+Z\right),\\[0.3cm] 
					\dot{Z}&=&\dfrac{\varepsilon  F\left(\sqrt{b} r \sin \left(\sqrt{b} \theta \right)-a Z,a^2 Z-b r \cos \left(\sqrt{b} \theta \right),r \cos \left(\sqrt{b} \theta \right)+Z\right)}{a^2+b}\\
                    & &-a Z. 
				\end{array}\right.\end{equation}
			Changing the independent variable $t$ of the system \eqref{spercil} for the variable $\theta$, we obtain the equivalent 2-dimensional system 
			\begin{equation}\left\{\label{sfinal}\begin{array}{lll}
					r'&=&\dfrac{\varepsilon  \left(\sqrt{b} \cos \left(\sqrt{b} \theta \right)-a \sin \left(\sqrt{b} \theta \right)\right)}{\sqrt{b} \left(a^2+b\right)}\\[0.3cm]
					& & F\left(\sqrt{b} r \sin \left(\sqrt{b} \theta \right)-a Z,a^2 Z-b r \cos \left(\sqrt{b} \theta \right),r \cos \left(\sqrt{b} \theta \right)+Z\right),\\[0.3cm]
					Z'&=&\dfrac{\varepsilon  \left(a^2 Z \cos \left(\sqrt{b} \theta \right)+a \sqrt{b} Z \sin \left(\sqrt{b} \theta \right)-b r\right)}{b r \left(a^2+b\right)}\\[0.3cm]
					& & F\left(\sqrt{b} r \sin \left(\sqrt{b} \theta \right)-a Z,a^2 Z-b r \cos \left(\sqrt{b} \theta \right),r \cos \left(\sqrt{b} \theta \right)+Z\right)\\
                    & &+a Z,
				\end{array}\right.\end{equation}
			
			\noindent where the prime ($'$) denotes the derivative with respect to $\theta$.
			
			Thus, if we use the notation ${\bf x}=(r,Z)$, we have that the system (\ref{sfinal}) can be written in the form
			$${\bf x}'(\theta)=F_0(\theta,{\bf x})+\varepsilon F_1(\theta,{\bf x})+\varepsilon^{2} F_2(\theta,{\bf x},\varepsilon),$$
			with $F_0,F_1:\mathbb{R}\times \Omega \rightarrow \Omega$ e $F_2:\mathbb{R}\times \Omega \times (-\varepsilon_0,\varepsilon_0) \rightarrow \Omega$, where $\Omega=\{(r,Z)|\; Z\in \mathbb{R},\; r>0\}$, $F_0(\theta,{\bf x})=\left(0,a Z\right)$ and
			\begin{align*}
				F_1(\theta,{\bf x})=&\left( \dfrac{\left(\sqrt{b} \cos \left(\sqrt{b} \theta \right)-a \sin \left(\sqrt{b} \theta \right)\right)}{\sqrt{b} \left(a^2+b\right)}\right. \\[0.2cm]
				& F\left(\sqrt{b} r \sin \left(\sqrt{b} \theta \right)-a Z,a^2 Z-b r \cos \left(\sqrt{b} \theta \right),r \cos \left(\sqrt{b} \theta \right)+Z\right),\\[0.2cm]
				& \dfrac{\left(a^2 Z \cos \left(\sqrt{b} \theta \right)+a \sqrt{b} Z \sin \left(\sqrt{b} \theta \right)-b r\right) }{b r \left(a^2+b\right)}\\[0.2cm]
				&\left. F\left(\sqrt{b} r \sin \left(\sqrt{b} \theta \right)-a Z,a^2 Z-b r \cos \left(\sqrt{b} \theta \right),r \cos \left(\sqrt{b} \theta \right)+Z\right)\right) . \\
			\end{align*}
			
			Consider the subset 
			$\mathcal{Z}=\{z_r=(r,0)|\; r>0\}$
			of $\Omega$. The general solution of the system $(\ref{sfinal})|_{\varepsilon=0}$ is ${\bf x}=(c_2,c_1e^{at})$, thus, the solution passing through the point $z_r$ is ${\bf x}(\theta, z_r)=(r,0)$, which is constant, hence $\frac{2\pi}{\sqrt{b}}$-periodic in $\theta$.  
			
			We have that the linearization of the unperturbed system $(\ref{sfinal})|_{\varepsilon=0}$ along the solutions of $\mathcal{Z}$ is 
			\begin{equation}\label{matrix}
				\left( \begin{array}{c}
					r'\\
					Z'
				\end{array}\right)=
				\left( \begin{array}{cc}
					0 & 0\\
					0 & a
				\end{array}\right)
				\left( \begin{array}{c}
					r\\
					Z
				\end{array}\right).
			\end{equation}	
			The fundamental matrix of (\ref{matrix}) and its inverse are 
			$$M_{z_r}(\theta)=\left(
			\begin{array}{cc}
				1 & 0 \\
				0 & e^{a\theta} \\
			\end{array}
			\right) 
			\quad \mbox{and} \quad  
			M_{z_r}^{-1}(\theta)\left(
			\begin{array}{cc}
				1 & 0 \\
				0 & e^{-a \theta} \\
			\end{array}
			\right).$$
			Thus,
			$$M_{z_r}^{-1}(0)-M_{z_r}^{-1}\left(\frac{2\pi}{\sqrt{b}}\right)=\left(
			\begin{array}{cc}
				0 & 0 \\
				0 & 1-e^{-\frac{2 \pi  a}{\sqrt{b}}} \\
			\end{array}
			\right).
			$$
			Taking $V=(r_1,r_2)$ a subset of $\mathbb{R}$ and $\beta : [r_1,r_2]\rightarrow \mathbb{R}$ the constant function $0$, we get that the system (\ref{sfinal}) satisfies the conditions (a) and (b) of Theorem \ref{averaging}. 
			
			Now, according to Theorem \ref{averaging}, we must study the zeros in $V$ of the equation $\mathcal{F}(r)=0$, where
			$$\mathcal{F}(r)=\xi \left(\dfrac{1}{2\pi}\int_{0}^{\frac{2\pi}{\sqrt{b}}} M_{z_r}^{-1}(\theta)F_1(\theta,{\bf x}(\theta,z_r))\, d\theta \right), $$
			with
			\begin{align*}
				M_{z_r}^{-1}(\theta)F_1(\theta,{\bf x}(\theta,z_r))=&
				\left( \dfrac{\left(\sqrt{b} \cos \left(\sqrt{b} \theta \right)-a \sin \left(\sqrt{b} \theta \right)\right)}{\sqrt{b} \left(a^2+b\right)}\right. \\[0.2cm]
				& F\left(\sqrt{b} r \sin \left(\sqrt{b} \theta \right),-b r \cos \left(\sqrt{b} \theta \right),r \cos \left(\sqrt{b} \theta \right)\right),\\[0.2cm]
				& -F\left(\sqrt{b} r \sin \left(\sqrt{b} \theta \right),-b r \cos \left(\sqrt{b} \theta \right),r \cos \left(\sqrt{b} \theta \right)\right)\\
                &\left.\frac{e^{-a \theta }}{a^2+b}\right). 
			\end{align*}
			Hence, 
			\begin{align*}
				\mathcal{F}(r)&=\dfrac{1}{2\pi (a^2+b)}\int_0^{\frac{2\pi}{\sqrt{b}}} \left(\sqrt{b} \cos \left(\sqrt{b} \theta \right)-a \sin \left(\sqrt{b} \theta \right)\right)\\
				&\quad F\left(\sqrt{b} r \sin \left(\sqrt{b} \theta \right),-b r \cos \left(\sqrt{b} \theta \right),r \cos \left(\sqrt{b} \theta \right)\right) \, d\theta.
			\end{align*}
			Therefore, we conclude that the proof of the result follows from Theorem \ref{averaging}.
		\end{proof}
		
		\begin{lemma}\label{intzero}
			Consider
			$$\int_{0}^{\frac{2\pi}{\sqrt{b}}}\sin^m\left(\sqrt{b}\theta\right)\cos^n\left(\sqrt{b}\theta\right)\,d\theta,$$
			with $b>0$. When $n$ or $m$ is odd, the integral is zero.  
		\end{lemma}
		
		\begin{proof}
			When $n=2k+1$, we have 
        \begin{align*}
		&\int_{0}^{\frac{2\pi}{\sqrt{b}}}\sin^m\left(\sqrt{b}\theta\right)\cos^{2k+1}\left(\sqrt{b}\theta\right)\,d\theta=\int_{0}^{\frac{2\pi}{\sqrt{b}}}\sin^m\left(\sqrt{b}\theta\right)\left(1-\sin^2\left(\sqrt{b}\theta\right)\right)^{k}\cos\left(\sqrt{b}\theta\right)\,d\theta.
        \end{align*}
			Considering the substitution $u=\sin\left(\sqrt{b}\theta\right)$, we get
			$$\dfrac{1}{\sqrt{b}}\int_0^0u^m(1-u^2)^k\,du=0.$$
			Now, when $m=2k+1$, we have 
        \begin{align*}
		&\int_{0}^{\frac{2\pi}{\sqrt{b}}}\sin^{2k+1}\left(\sqrt{b}\theta\right)\cos^{n}\left(\sqrt{b}\theta\right)\,d\theta=\int_{0}^{\frac{2\pi}{\sqrt{b}}}\sin\left(\sqrt{b}\theta\right)\left(1-\cos^2\left(\sqrt{b}\theta\right)\right)^k\cos\left(\sqrt{b}\theta\right)\,d\theta.
        \end{align*}
			Considering the substitution $v=\cos\left(\sqrt{b}\theta\right)$, we obtain
			$$-\dfrac{1}{\sqrt{b}}\int_0^0v^n(1-v^2)^k\,du=0.$$
		\end{proof}

		\begin{remark}\label{gamma}
			Remember that if the real part of the complex number $z$ is strictly positive, we define the gamma function
			$$ \Gamma(z)=\int _{0}^{\infty }x^{z-1}e^{-x}\,dx.$$
			Moreover, for non-negative integer values of $n$, we have
			$$\Gamma(n)=(n-1)! \quad \mbox{and} \quad \Gamma\left( \frac{1}{2}+n\right)=\frac{(2n-1)!!}{2^n}\sqrt{\pi}.$$
			For more details see \cite{gamma,artingamma}. 
		\end{remark}
		
		\begin{remark}\label{beta}
			Remember that for $n,m \in \mathbb{R}$ we define the beta function
			$$B(n,m)=\int_0^1x^{n-1}(1-x)^{m-1}\;dx,$$
			and this function is related with the gamma function by the equality 
			$$B(n,m)=\frac{\Gamma(n)\Gamma(m)}{\Gamma(n+m)}.$$
			Moreover, considering the substitution $x=\sin^2 (\theta)$, we get
        \begin{align*}
		B(n,m)&=2\int_0^{\frac{\pi}{2}}(\sin^2(\theta))^{n-1}(1-\sin^2(\theta))^{m-1}\sin(\theta)\cos(\theta)\;d\theta\\
        &= 2\int_0^{\frac{\pi}{2}}\sin^{2n-1}(\theta)\cos^{2m-1}(\theta)\;d\theta.
        \end{align*}
			For more details see \cite{gamma,artingamma}. 
		\end{remark}
		
		\begin{lemma}\label{betaint}
			Assume $b>0$. When $n$ is even, 
			$$\int_{0}^{\frac{2\pi}{\sqrt{b}}}\sin^n\left(\sqrt{b}\theta\right)\cos^m\left(\sqrt{b}\theta\right)\,d\theta = \dfrac{2\pi(m-1)!!(n-1)!!}{\sqrt{b}(m+n)!!}.$$
		\end{lemma}
		
		\begin{proof}
			Considering the substitution $u= \sqrt{b}\theta$, we get
			$$\int_{0}^{\frac{2\pi}{\sqrt{b}}}\sin^n\left(\sqrt{b}\theta\right)\cos^m\left(\sqrt{b}\theta\right)\,d\theta = \frac{1}{\sqrt{b}}\int_{0}^{2\pi}\sin^n(u)\cos^m(u)\,du.$$
			We note that 
			$$\sin^n\left(u+\frac{\pi}{2}\right)\cos^m\left(u+\frac{\pi}{2}\right)= \sin^n(u)\cos^m(u),$$
			because
			$$\sin\left(u+\frac{\pi}{2}\right)=\cos(u)\cos\left(\frac{\pi}{2}\right)-\sin(u)\sin\left(\frac{\pi}{2}\right)$$
			and
			$$\cos\left(u+\frac{\pi}{2}\right)=\sin(u)\cos\left(\frac{\pi}{2}\right)+\sin\left(\frac{\pi}{2}\right)\cos(u),$$
			that is, the term $\sin^n(u)\cos^m(u)$ of the integral is $\frac{\pi}{2}$-periodic.
			Thus, dividing the interval of integration into four subintervals of size $\frac{\pi}{2}$, we have that
			$$ \frac{1}{\sqrt{b}}\int_{0}^{2\pi}\sin^n(u)\cos^m(u)\,du =  \frac{4}{\sqrt{b}}\int_{0}^{\frac{\pi}{2}}\sin^n(u)\cos^m(u)\,du.$$

			By Remark \ref{beta}, we have
        \begin{align*}
		\int_{0}^{\frac{2\pi}{\sqrt{b}}}\sin^n\left(\sqrt{b}\theta\right)\cos^m\left(\sqrt{b}\theta\right)\,d\theta&=\frac{2}{\sqrt{b}}B\left(\frac{n+1}{2},\frac{m+1}{2}\right)=\frac{2}{\sqrt{b}}\frac{\Gamma\left(\frac{n+1}{2}\right)\Gamma\left(\frac{m+1}{2}\right)}{\Gamma\left(\frac{n+1}{2}+\frac{m+1}{2}\right)}
        \end{align*}
			and, by Remark \ref{gamma}, we get 
			$$\Gamma\left(\frac{n+1}{2}\right)\Gamma\left(\frac{m+1}{2}\right)=\frac{(n-1)!! (m-1)!!}{2^{n+m}}\pi, $$
			$$\Gamma\left(\frac{n+1}{2}+\frac{m+1}{2}\right)=\left(\frac{m+n}{2}\right)!=\frac{(m+n)!!}{2^{m+n}}.$$
			
			Therefore, $$\int_{0}^{\frac{2\pi}{\sqrt{b}}}\sin^n\left(\sqrt{b}\theta\right)\cos^m\left(\sqrt{b}\theta\right)\,d\theta=\frac{2\pi}{\sqrt{b}}\dfrac{(m-1)!!(n-1)!!}{(m+n)!!}.$$
		\end{proof}

		With the aid of previous lemmas, we are able to finish the proof of Theorem \ref{teo2}. 
		
		\begin{proof}[Proof of Theorem \ref{teo2} for $\varepsilon \neq 0$]
			Initially, we are going to calculate 
        \begin{align*}    
		\tilde{\mathcal{F}}(r)=&\int_0^{\frac{2\pi}{\sqrt{b}}} \left(\sqrt{b} \cos \left(\sqrt{b} \theta \right)- a \sin \left(\sqrt{b} \theta \right)\right)\\
        &F \left(\sqrt{b} r \sin \left(\sqrt{b} \theta \right),-b r \cos \left(\sqrt{b} \theta \right),r \cos \left(\sqrt{b} \theta \right)\right) d\theta.
        \end{align*}
			
			We know that $F(z,\dot{z},\ddot{z})$ is a polynomial of degree $n$, then we can write
			$$F(z,x,y)=\sum_{i+j+k=0}^{n}a_{ijk}x^iy^jz^k.$$
			Hence,
			\begin{align*}
				&\quad\displaystyle F\left(\sqrt{b} r \sin \left(\sqrt{b}\theta \right),-b r \cos \left(\sqrt{b}\theta \right),r \cos \left(\sqrt{b}\theta \right)\right)\\
				&=\displaystyle \sum_{i+j+k=0}^{n}a_{ijk}\left(\sqrt{b}r \sin \left(\sqrt{b}\theta \right)\right)^i\left(-b r \cos \left(\sqrt{b}\theta \right)\right)^j\left(r \cos \left(\sqrt{b}\theta \right)\right)^k\\
				&=\displaystyle \sum_{i+j+k=0}^{n}a_{ijk}(-1)^j b^{\frac{i}{2}+j} r^{i+j+k} \sin^i\left(\sqrt{b}\theta\right)\cos^{j+k}\left(\sqrt{b}\theta\right).
			\end{align*}
			In this way,
			{\allowdisplaybreaks
				\begin{flalign*}
					\tilde{\mathcal{F}}(r) &= \sqrt{b}\int_0^{\frac{2\pi}{\sqrt{b}}} \!\!\!\cos \left(\sqrt{b}\theta \right)\!  F\left(\sqrt{b} r \sin \left(\sqrt{b} \theta \right),-b r \cos \left(\sqrt{b} \theta \right),r \cos \left(\sqrt{b} \theta \right)\right) d\theta\\[0.2cm]  
					&\quad -a \int_0^{\frac{2\pi}{\sqrt{b}} } \!\!\!\sin \left(\sqrt{b}\theta \right)\!  F\left(\sqrt{b} r \sin \left(\sqrt{b} \theta \right),-b r \cos \left(\sqrt{b} \theta \right),r \cos \left(\sqrt{b} \theta \right)\right) d\theta\\[0.2cm]  
					&= \sum_{i+j+k=0}^{n}\!\!\!a_{ijk}(-1)^jb^{\frac{i}{2}+j}r^{i+j+k} \left[ \sqrt{b}\int_0^{\frac{2\pi}{\sqrt{b}} }\sin^i\left(\sqrt{b}\theta\right)\cos^{j+k+1}\left(\sqrt{b}\theta\right)\, d\theta \right.\\[0.2cm]  
					&\quad \left.-a \int_0^{\frac{2\pi}{\sqrt{b}}}\sin^{i+1}\left(\sqrt{b}\theta\right)\cos^{j+k}\left(\sqrt{b}\theta\right) \, d\theta\right].   
				\end{flalign*}
			}
			
			Suppose that $i+j+k$ is even. With this, we have that $i+j+k+1$ is odd and, furthermore,
			\begin{itemize}
				\item if $i$ is even we get that $i+1$ is odd, $j+k$ is even and $j+k+1$ is odd;
				\item if $i$ is odd we get that $i+1$ is even, $j+k$ is odd and $j+k+1$ is even.
			\end{itemize}
			Hence, in this case, from Lemma  \ref{intzero},
			$$\int_0^{\frac{2\pi}{\sqrt{b}} }\sin^i\left(\sqrt{b}\theta\right)\cos^{j+k+1}\left(\sqrt{b}\theta\right)=0$$
            and
            $$ \int_0^{\frac{2\pi}{\sqrt{b}} }\sin^{i+1}\left(\sqrt{b}\theta\right)\cos^{j+k}\left(\sqrt{b}\theta\right) \, d\theta=0.$$
			Thus, we have
			\begin{align*}
				\tilde{\mathcal{F}}(r)=&\sum_{\substack{i+j+k=0\\ i+j+k=2m+1}}^{n}a_{ijk}(-1)^jb^{\frac{i}{2}+j}r^{i+j+k}\left[\sqrt{b} \int_0^{\frac{2\pi}{\sqrt{b}}}\sin^i\left(\sqrt{b}\theta\right)\cos^{j+k+1}\left(\sqrt{b}\theta\right)\, d\theta\right. \\[0.2cm] 
				&\left. -a \int_0^{\frac{2\pi}{\sqrt{b}} }\sin^{i+1}\left(\sqrt{b}\theta\right)\cos^{j+k}\left(\sqrt{b}\theta\right) \, d\theta\right].
			\end{align*}

			Now, consider that $i+j+k$ is odd. Then, we have that $i+j+k+1$ is even and, furthermore,
			\begin{itemize}
				\item if $i$ is even we get that $i+1$ is odd, $j+k$ is odd and $j+k+1$ is even. Hence, by Lemma \ref{intzero} and Lemma  \ref{betaint}, we have 
				$$\int_0^{\frac{2\pi}{\sqrt{b}} }\sin^{i+1}\left(\sqrt{b}\theta\right)\cos^{j+k}\left(\sqrt{b}\theta\right) \, d\theta=0$$
				and
				$$\int_0^{\frac{2\pi}{\sqrt{b}} }\sin^i\left(\sqrt{b}\theta\right)\cos^{j+k+1}\left(\sqrt{b}\theta\right)\, d\theta = \frac{2\pi}{\sqrt{b}}\dfrac{(i-1)!!(j+k)!!}{(i+j+k+1)!!}\neq 0, $$
				respectively;
				\item if $i$ is odd we get that $i+1$ is even, $j+k$ is even and $j+k+1$ is odd. Hence, by Lemma \ref{intzero} and Lemma  \ref{betaint}, we have 
				$$\int_0^{\frac{2\pi}{\sqrt{b}}}\sin^i\left(\sqrt{b}\theta\right)\cos^{j+k+1}\left(\sqrt{b}\theta\right)\, d\theta$$
				and
				$$ \int_0^{\frac{2\pi}{\sqrt{b}} }\sin^{i+1}\left(\sqrt{b}\theta\right)\cos^{j+k}\left(\sqrt{b}\theta\right) \, d\theta = \frac{2\pi}{\sqrt{b}}\dfrac{i!!(j+k-1)!!}{(i+j+k+1)!!} \neq 0,$$
				respectively.
			\end{itemize}
			Thus, with these information, we get 
			\begin{align*}
				\tilde{\mathcal{F}}(r)&=\!\!\!\!\!\sum_{\substack{i+j+k=0\\ i+j+k=2m+1\\i=2p}}^{n}\!\!\!\!\!a_{ijk}(-1)^jb^{\frac{i}{2}+j+\frac{1}{2}}r^{i+j+k}  \int_0^{\frac{2\pi}{\sqrt{b}} }\sin^i\left(\sqrt{b}\theta\right)\cos^{j+k+1}\left(\sqrt{b}\theta\right)\, d\theta\\[0.2cm]
				&-a\left(\!\!\sum_{\substack{i+j+k=0\\ i+j+k=2m+1\\i=2p+1}}^{n}\!\!\!\!\!a_{ijk}(-1)^jb^{\frac{i}{2}+j}r^{i+j+k}\!\!\int_0^{\frac{2\pi}{\sqrt{b}} }\sin^{i+1}\left(\sqrt{b}\theta\right)\cos^{j+k}\left(\sqrt{b}\theta\right) \, d\theta\!\!\!\right) \\[0.2cm]
				&=\frac{2 \pi}{\sqrt{b}}\sum_{\substack{i+j+k=0\\ i+j+k=2m+1\\i=2p}}^{n}a_{ijk}(-1)^jb^{\frac{i}{2}+j+\frac{1}{2}}r^{i+j+k}  \dfrac{(i-1)!!(j+k)!!}{(i+j+k+1)!!}\\[0.2cm]
				&\quad-\frac{2\pi a}{\sqrt{b}}  \sum_{\substack{i+j+k=0\\ i+j+k=2m+1\\i=2p+1}}^{n}a_{ijk}(-1)^jb^{\frac{i}{2}+j}r^{i+j+k}\dfrac{i!!(j+k-1)!!}{(i+j+k+1)!!}. 
			\end{align*}

			Since $\mathcal{F}(r)= \dfrac{1}{2\pi (a^2+b)}\mathcal{\tilde{F}}(r)$, we obtain
			\begin{align*}
				\mathcal{F}(r)&=\frac{1}{\sqrt{b}(a^2+b)}\sum_{\substack{i+j+k=0\\ i+j+k=2m+1\\i=2p}}^{n}a_{ijk}(-1)^jb^{\frac{i}{2}+j+\frac{1}{2}}r^{i+j+k}  \dfrac{(i-1)!!(j+k)!!}{(i+j+k+1)!!}\\[0.2cm]
				&\quad-\frac{a}{\sqrt{b}(a^2+b)}  \sum_{\substack{i+j+k=0\\ i+j+k=2m+1\\i=2p+1}}^{n}a_{ijk}(-1)^jb^{\frac{i}{2}+j}r^{i+j+k}\dfrac{i!!(j+k-1)!!}{(i+j+k+1)!!}. 
			\end{align*}
			
			To finish, if $n$ is even, by Theorem \ref{descarte}, we get that the maximum number of positive real roots counting multiplicity of $\mathcal{F}(r)$ is $\dfrac{n}{2}-1$. If $n$ is odd, by Theorem \ref{descarte}, we obtain that the maximum number of positive real roots counting multiplicity of $\mathcal{F}(r)$ is $\dfrac{n+1}{2}-1$. 
			
			Therefore, from Lemma \ref{bifurcate} we conclude that if $n$ is odd, $\frac{n-1}{2}$ is the maximum number of limit cycles of the system (\ref{sisper}) that can bifurcate from the periodic orbits of the plane $\alpha$ of the system $(\ref{sisper})|_{\varepsilon =0}$, and if $n$ is even, $\frac{n-2}{2}$ is the maximum number of limit cycles of the system (\ref{sisper}) that can bifurcate from the periodic orbits of the plane $\alpha$ of the system $(\ref{sisper})|_{\varepsilon =0}$.
		\end{proof}

\section{Declarations}

\paragraph{\bf Conflict of interest:} The authors declare that they have no conflict of interest.
\paragraph{\bf Availability of data and material:} Not applicable.
\paragraph{\bf Code availability:} Not applicable.

\section{Acknowledgements}

\noindent Mayara Caldas was financed in part by the Coordenação de Aperfeiçoamento de Pessoal de Nível Superior - Brasil (CAPES) - Finance Code 001. Ricardo Martins was partially supported by FAPESP grants 2021/08031-9 and 2018/03338-6, CNPq grants 315925/2021-3 and 434599/2018-2, and Unicamp/Faepex grant 2475/21.

\footnotesize


\bibliographystyle{plain}
\bibliography{bibli}

\end{document}